\documentclass[10pt]{amsart}
\usepackage{amsmath, amsthm, amssymb, euscript, enumerate}
\usepackage{pxfonts, graphicx, subfigure, epsfig}
\usepackage{mathtools, appendix, accents, color}
\usepackage[pdftex]{hyperref}
\usepackage[latin1]{inputenc}
\usepackage{tikz}
\usetikzlibrary{shapes, arrows}
\usepackage{tabulary}
\usepackage{booktabs}
\usepackage{esvect}

\newcommand{\R}{{\mathbb R}}

\newcommand{\T}{{\mathbb T}}
\newcommand{\Z}{{\mathbb Z}}

\newcommand{\cS}{{\mathcal S}}

\newcommand{\e}{\varepsilon}
\newcommand{\vp}{\varphi}

\newcommand{\p}{\partial}

\newcommand{\ra}{\rightarrow}

\newcommand{\norm}[1]{\left\Arrowvert {#1}\right\Arrowvert}

\newcommand{\osc}{\operatornamewithlimits{osc}}

\newcommand{\tr}{\operatorname{tr}}

\theoremstyle{plain}
\newtheorem{theorem}{Theorem}[section]

\newtheorem{corollary}[theorem]{Corollary}

\newtheorem{lemma}[theorem]{Lemma}
\newtheorem{proposition}[theorem]{Proposition}

\theoremstyle{definition}

\newtheorem{example}[theorem]{Example}

\theoremstyle{remark}

\newtheorem{case}{Case}
	\numberwithin{case}{theorem}

\newtheorem{remark}[theorem]{Remark}

\numberwithin{equation}{section}

\title[Higher Order Convergence Rates: Oscillatory Initial Data]{Higher Order Convergence Rates in Theory of Homogenization II: Oscillatory Initial Data}
\author{Sunghan Kim}
\address{Department of Mathematical Sciences, Seoul National University, Seoul 08826, Korea}
\email{sunghan290@snu.ac.kr}
\author{Ki-Ahm Lee}
\address{Department of Mathematical Sciences, Seoul National University, Seoul 08826, Korea
\& Korea Institute for Advanced Study, Seoul 02455, Korea}
\email{kiahm@snu.ac.kr}
\thanks {S. Kim was supported by National Research Foundation of Korea under Project Number NRF-2014H1A8A1021602. K.-A. Lee was supported by Samsung Science and Technology Foundation under Project Number SSTF-BA1701-03. K.-A. Lee also holds a joint appointment with the Research Institute of Mathematics of Seoul National University.}

\begin{document}

\maketitle
\begin{abstract} 
We establish higher order convergence rates in periodic homogenization of fully nonlinear uniformly parabolic Cauchy problems accompanied with rapidly oscillating initial data. Such result is new even for linear problems. Here we construct higher order initial layer and interior correctors, which describe the oscillatory behavior near the initial and interior time zone of the domain. To construct higher order correctors, we develop a regularity theory in macroscopic scales, and prove an exponential decay estimate for initial layer correctors. 

The higher order expansion requires an iteration process: successively correcting the initial layer, then the interior. This leads to a more complicated asymptotic expansion, as compared to the non-oscillating data case, and this complexity is present even in the linear case. 

A notable observation for fully nonlinear operators is that even if the given operator is space-time periodic, the interior correctors become aperiodic in the time variable as we proceed with the iteration process. Moreover, each interior corrector of higher order is paired with a space-time periodic version, and the difference between the two decays exponentially fast with time. 
\end{abstract}
\maketitle
\tableofcontents


\section{Introduction}\label{section:intro}

Homogenization of a rapidly oscillating data on a lower dimensional object has been an important topic in analysis, not only because of its major implications in other fields, but also because of the mathematical challenges it introduces. One of the central issues in such a homogenization problem is quantifying the rate of its process. Although the homogenization problem was solved for various settings (e.g., see \cite{KLS} for linear Dirichlet problems, \cite{F} for nonlinear Dirichlet problems, \cite{CKL} for nonlinear Neumann problems, etc.), there is only a small number of results concerning convergence rate. Moreover, most of those results are focused on linear Dirichlet problems of divergence type (e.g., see \cite{AKMP,GM} for recent developments in this direction). As a matter of fact, sharp estimates are still open for Dirichlet or Neumann problems of non-divergence type, even for linear equations. 

In this article, we shall study the higher order convergence rate for a homogenization problem with a highly oscillating data on a lower dimensional object. Our analysis will be based on the theory of viscosity solutions and cover both linear and fully nonlinear equations. 

In order to simplify the structure of lower dimensional objects, this paper will treat a space-time periodic Cauchy problem. In this case, the corresponding lower dimensional object is the initial time layer, which is a hyperplane that persists the periodic lattice structure in the interior of the domain. In this way, we can focus on effects solely coming from a given rapidly oscillating boundary data, but avoid some sophisticated issues that arise when the boundary layer has a qualitatively different oscillating pattern from that of the interior (e.g., when the normal direction is a totally irrational vector, the oscillating pattern of a periodic data restricted to the boundary layer becomes quasi-periodic), or when the boundary layer is bended and one has to take the curvature effect into account.  

Nevertheless, the problem under such a simple structure already exhibits some challenges, as we aim to construct higher order correctors for the given homogenization process. It should be remarked that such challenges are present even in the context of linear equations. In the framework of fully nonlinear equations, we face more complexity, since the nonlinearity of the given operator changes the oscillating pattern of the higher order correctors.


\subsection{Main Result}\label{subsection:result}

Let us consider the following second order parabolic Cauchy problem, 
\begin{equation}\label{eq:ue-pde}
\begin{dcases}
u_t^\e = \frac{1}{\e^2} F \left( \e^2 D^2 u^\e , x, t, \frac{x}{\e}, \frac{t}{\e^2} \right) & \text{in }\R^n\times(0,T),\\
u^\e(x,0) = g\left( x, \frac{x}{\e} \right) & \text{on }\R^n,
\end{dcases}
\end{equation}
where $F$ is a smooth, concave, uniformly parabolic functional periodic in $(\e^{-1}x,\e^{-2}t)$, and $g$ is a smooth data periodic in $\e^{-1}x$. The precise assumptions will be made later in Section \ref{section:assump}. If $F$ is a linear functional in the Hessian variable, the problem reads
\begin{equation}\label{eq:ue-pde-lin}
\begin{dcases}
u_t^\e = \tr \left( A \left( x, t, \frac{x}{\e}, \frac{t}{\e^2} \right) D^2 u^\e  \right) & \text{in }\R^n\times(0,T),\\
u^\e(x,0) = g\left( x, \frac{x}{\e} \right) & \text{on }\R^n.
\end{dcases}
\end{equation}
Our main result is stated as follows. 

\begin{theorem}\label{theorem:higher} Assume that $F:\cS^n\times \R^n\times[0,T]\times\T^n\times\T \ra \R$ and $g:\R^n\times\T^n\ra\R$ verify \eqref{eq:ellip-F} - \eqref{eq:Ck-C2a-g}, and let $\{u^\e\}_{\e>0}$ be the family of the bounded viscosity solutions to \eqref{eq:ue-pde}. Then for each integer $d\geq 0$, there exist sequences $\{\tilde{v}_{d,k}\}_{k=0}^\infty$, $\{\tilde{w}_{d,k}\}_{k=0}^\infty$ of functions on $\R^n\times[0,T]\times\T^n\times[0,\infty)$ and a sequence $\{\tilde{w}_{d,k}^\#\}_{k=0}^\infty$ of functions on $\R^n\times[0,T]\times\T^n\times\T$ such that for any integer $m\geq 2$, any $\e\leq\frac{1}{2}$, any $x\in\R^n$ and $0\leq t\leq T$, 
\begin{equation}\label{eq:higher}
\left| u^\e (x,t) - \sum_{d=0}^{\lfloor \frac{m}{2}\rfloor - 1} \sum_{k=0}^{m-2d} \e^{k + 2d} \left( \tilde{v}_{d,k}\left( x,t,\frac{x}{\e},\frac{t}{\e^2} \right) + \tilde{w}_{d,k}\left( x,t,\frac{x}{\e},\frac{t}{\e^2}\right) \right)  \right| \leq C_m\e^{m-1},
\end{equation}
and in particular for $c_m\e^2 |\log\e | \leq t\leq T$, 
\begin{equation}\label{eq:higher-log}
\left| u^\e (x,t) - \sum_{d=0}^{\lfloor \frac{m}{2} \rfloor - 1} \sum_{k=0}^{m-2d} \e^{k + 2d} \tilde{w}_{d,k}^\# \left( x,t,\frac{x}{\e},\frac{t}{\e^2}\right)  \right| \leq C_m\e^{m-1},
\end{equation}
where $c_m$ and $C_m$ depend only on $n$, $\lambda$, $\Lambda$, $\alpha$, $m$, $T$ and $K$.
\end{theorem}

Let us make a few remarks regarding Theorem \ref{theorem:higher}.

\begin{remark}\label{remark:higher-ini}
In Section \ref{subsection:bootstrap}, we observe that $\tilde{v}_{d,k}$ is of the form 
\begin{equation*}
\tilde{v}_{d,k} (x,t,y,s) = v_{d,k} (x,t,y,s) - \bar{v}_{d,k}(x,t),
\end{equation*}
and satisfies an exponential decay estimate as $s\ra\infty$. We shall call $v_{d,k}$ the $(d,k)$-th initial layer corrector and $\bar{g}_{d,k} = \bar{v}_{d,k}(\cdot,0)$ the $(d,k)$-th effective initial data. Owing to the exponential decay estimate for $\tilde{v}_{d,k}$, we obtain the higher order convergence rate \eqref{eq:higher-log} away from initial time zone $t\lesssim \e^2$. 
\end{remark}

\begin{remark}\label{remark:higher-int}
Moreover, $\tilde{w}_{d,k}$ is of the form 
\begin{equation*}
\tilde{w}_{d,k} (x,t,y,s) = w_{d,k}(x,t,y,s) + \bar{u}_{d,k}(x,t),
\end{equation*}
where $w_{d,k}$ and $\bar{u}_{d,k}$ will be called the interior corrector and respectively the effective limit profile of order $k+2d$. Furthermore, $w_{d,k}$ will be paired with a space-time periodic function $w_{d,k}^\#$ such that 
\begin{equation*}
\tilde{w}_{d,k}^\# (x,t,y,s) = w_{d,k}^\#(x,t,y,s) + \bar{u}_{d,k}(x,t).
\end{equation*}
Here $w_{d,k}^\#$ will play the role of the interior corrector in purely periodic homogenization problems. In particular, we shall observe that $\tilde{w}_{d,k} = \tilde{w}_{d,k}^\#$ when the interior equation is linear, i.e., $F(P,x,t,y,s) = \tr(A(x,t,y,s) P)$ for some matrix-valued mapping $A$.

In addition, $\tilde{w}_{d,k}$ and $\tilde{w}_{d,k}^\#$ for $k=0,1$ will turn out to be constant in the fast variables $(y,s)$, so the interior error estimate \eqref{eq:higher-log} is of the form $u^\e - \bar{u}_{0,0} - \bar{u}_{0,1} - \e^2 \tilde{w}_{0,2} - \cdots$. This indicates that there is no rapid oscillation in the interior (with respect to the microscopic scales) up to order $\e$. 
\end{remark}

\begin{remark}\label{remark:higher-assump}
In Theorem \ref{theorem:higher} and to the rest of this paper, we assume that $F$ is concave in its matrix variable. Such an assumption is made to have $C^{2,\alpha}$ correctors in fast variables $(y,s) = (\e^{-1}x,\e^{-2}t)$, and smooth limit profiles in slow variables $(x,t)$. These are essential, at least in our approach, to establish higher order convergence rates, since it requires accurate error correction at each order of $\e$. 

Let us also address that the effective problem of \eqref{eq:ue-pde} will turn out to be linear in the Hessian variable. However, this does not make the problem easier in the sense that we will still have strong nonlinear coupling effect near the initial time layer when we construct higher order correctors. The reason that our main problem \eqref{eq:ue-pde} has quadratic scaling in the Hessian variable (i.e., $\e^{-2}F(\e^2 P,x,t,y,s)$) is used to derive smooth correctors in slow variables. We shall discuss more on this issue later. 
\end{remark}


\subsection{Historical Background}\label{subsection:background}

Periodic homogenization of \eqref{eq:ue-pde} (or more generally \eqref{eq:ue-pde-nl}) is rigorously justified in \cite{AB} and \cite{M1}; see also the references therein, and \cite{Ish} for first order fully nonlinear equations as well as \cite{ABM} for iterated homogenization. There is a wide range of literature on the rate of convergence regarding the homogenization problems of type \eqref{eq:ue-pde}, provided that the initial data is non-oscillatory; that is, $g$ is independent on its second argument. Recent development can be found, for instance, in \cite{JK}, \cite{M2} and \cite{CM} using continuous dependence estimates, \cite{Ich} based on a different approach, and \cite{L}, \cite{LS} in stationary ergodic settings; see also the references therein for classical results in this regard. 

Higher order convergence rate in the theory of homogenization has been studied in various settings. We refer to \cite{CS}, \cite{KMP} for divergence type elliptic equations, \cite{MS} for perforated domains with mixed boundary conditions, \cite{CE} for Maxwell equations, \cite{ABDW} for wave equations, \cite{HO} for some numerical results, and also the references therein. Recently, the authors proved in \cite{KL1,KL2} higher order convergence rate for non-divergence type elliptic equations as well as viscous Hamilton-Jacobi equations. 

As far as we know, however, there has not yet been any result on (higher order) convergence rate in homogenization of \eqref{eq:ue-pde}. Here we also achieve higher order convergence rate (Proposition \ref{proposition:higher-nosc}) for uniformly parabolic equations with non-oscillatory initial data, that is, 
\begin{equation}\label{eq:ube-pde}
\begin{dcases}
u_t^\e =F \left( D^2 u^\e , x,t,\frac{x}{\e}, \frac{t}{\e^2} \right) & \text{in }\R^n\times(0,T),\\
u^\e(x,0) = g(x) & \text{on }\R^n;
\end{dcases}
\end{equation}
this will turn out to be a simple application of our main theorem, Theorem \ref{theorem:higher}. Moreover, we achieve a quantitative error estimate (Proposition \ref{proposition:opt}) in the following homogenization problem,
\begin{equation}\label{eq:ue-pde-nl}
\begin{dcases}
u_t^\e =F \left( D^2 u^\e , x, t, \frac{x}{\e}, \frac{t}{\e^2}\right) & \text{in }\R^n\times(0,T),\\
u^\e(x,0) = g\left( x, \frac{x}{\e}\right) & \text{on }\R^n.
\end{dcases}
\end{equation}
The estimate depends on a particular structure of the given operator $F$ and the initial data $g$, and our estimate becomes sharp in some cases. 


\subsection{Heuristic Discussion and Main Difficulties}\label{subsection:discussion}

Let us sketch the higher order correction scheme, which is also visualized in Figure \ref{figure:iter}. Suppose that we aim to have higher order convergence rate of order $\e^{m-1}$. Then our iteration scheme will consist of $\lfloor \frac{m}{2} \rfloor$ steps of successive correction near the initial layer and in the interior, where $\lfloor a \rfloor$ is the greatest integer less than or equal to $a$. 

At the $d$-th step, with $0\leq d\leq \lfloor \frac{m}{2} \rfloor$, we are given with a bounded solution $u_{d,m}^\e = u_{d,m}^\e(x,t)$ to a uniformly parabolic Cauchy problem (of type \eqref{eq:udme-pde}) with a rapidly oscillating initial condition; here $u_{0,m}^\e = u^\e$, i.e., the bounded solution of \eqref{eq:ue-pde}. 

First we construct a sequence $\{v_{d,k} = v_{d,k}(x,t,y,s)\}_{k=0}^\infty$ of higher order initial layer correctors such that for each $k$, $v_{d,k}(x,t,y,s)$ is periodic in $y$ and converges exponentially fast to some $\bar{v}_{d,k} = \bar{v}_{d,k}(x,t)$, as $s\ra\infty$. In particular, we will have 
\begin{equation*}
\left| u_{d,m}^\e (x,t) -  \sum_{k=0}^{m-2d} \e^k \left( v_{d,k} \left(x,t,\frac{x}{\e},\frac{t}{\e^2}\right) - \bar{v}_{d,k}(x,t)\right) - \tilde{u}_{d,m}^\e (x,t) \right| = O(\e^{m-1}), 
\end{equation*} 
where $\tilde{u}_{d,m}^\e$ is a bounded solution to another uniformly parabolic Cauchy problem \eqref{eq:utme-pde}. What is important here is that due to the higher order correction near the initial time layer, the initial condition for $\tilde{u}_{d,m}^\e$ is no longer rapidly oscillating as $\e\ra 0$; more precisely, we shall have
\begin{equation*}
\tilde{u}_{d,m}^\e (x,0) = \sum_{k=0}^{m-2d} \e^k \bar{g}_{d,k} (x). 
\end{equation*}

On the other hand, the interior equation for $\tilde{u}_{d,m}^\e$ will be rapidly oscillating. For this reason, we construct a sequence $\{w_{d,k} = w_{d,k}(x,t,y,s)\}_{k=0}^\infty$ of higher order interior correctors and a sequence $\{\bar{u}_{d,k} = \bar{u}_{d,k}(x,t)\}_{k=0}^\infty$ of higher order effective limit profiles such that
\begin{equation*}
\left| \tilde{u}_{d,m}^\e(x,t) - \sum_{k=0}^{m-2d} \e^k \left( w_{d,k} \left(x,t,\frac{x}{\e},\frac{t}{\e^2}\right) + \bar{u}_{d,k} (x,t)\right) - \e^2 u_{d+1,m}^\e(x,t)\right| = O(\e^{m-1}).
\end{equation*} 
Here $u_{d+1,m}^\e$ is a bounded solution to yet another uniformly parabolic Cauchy problem (of type \eqref{eq:udme-pde}) subject to a rapidly oscillating initial condition. More specifically, we shall observe that 
\begin{equation*}
u_{d+1,m}^\e (x,0) = - \sum_{k=0}^{m-2d-2} \e^k w_{d,k+2} \left(x,0,\frac{x}{\e},0\right).
\end{equation*} 
The reason why the summation above does not involve $w_{d,0}$ and $w_{d,1}$ is that both terms turn out to be identically zero. The reason that we have $w_{d,0} = w_{d,1} = 0$ is due to the non-divergence structure of the Cauchy problem that $\tilde{u}_{d,m}^\e$ (as well as $u_{d,m}^\e$) solves. 

We shall observe that the Cauchy problem for $u_{d+1,m}^\e$ shares the same structure conditions with that of $u_{d,m}^\e$, in the sense that the interior equations are uniformly parabolic, smooth and concave in the Hessian variable, and that the initial conditions are rapidly oscillating and periodic in the fast variable. For this reason, we can iterate this correcting scheme, for $\lfloor \frac{m}{2} \rfloor$-times, and obtain an expansion of $u^\e = u_{0,m}^\e$ in terms of the higher order initial layer and interior correctors, up to an $O(\e^{m-1}$)-error. 

This sketches the higher order correction scheme. Let us stress that such an iterative structure is present in both linear and nonlinear problems. Let us also repeat that the improvement of $O(\e^2)$-error after each round of the higher order correction is due to the non-divergence structure of our Cauchy problem \eqref{eq:ue-pde}. That being said, we expect to have $O(\e)$-error improvement for divergence-type Cauchy problems, but we shall not discuss it further in this paper. 

In what follows, let us discuss the difficulties arising from the nonlinearity of the governing operator $F$. Let us denote by $F_{d,m}^\e = F_{d,m}^\e(P,x,t,y,s)$ and $\tilde{F}_{d,m}^\e = \tilde{F}_{d,m}^\e (P,x,t,y,s)$ the governing operators associated with the Cauchy problems for $u_{d,m}^\e$ and respectively $\tilde{u}_{d,m}^\e$. If $F$ is a linear operator so that $F(P,x,t,y,s) = \tr(A(x,t,y,s)P)$, then we shall see that 
\begin{equation*}
F_{d,m}^\e (P,x,t,y,s) = \tilde{F}_{d,m}^\e (P,x,t,y,s) = \tr(A(x,t,y,s)P) = F(P,x,t,y,s).
\end{equation*}
In other words, the interior equation remains the same in each step of the higher order correction for linear Cauchy problems \eqref{eq:ue-pde-lin}. 

In contrast, if $F$ is a fully nonlinear operator, we will have $F_{d,m}^\e$ and $\tilde{F}_{d,m}^\e$ all different for each $d$-th step. In particular, $\tilde{F}_{d,m}^\e$ will be given by  
\begin{equation*}
\tilde{F}_{d,m}^\e (P,x,t,y,s) = \e^{-2d-2}\left(  F \left( \e^{2d+2} P + Y_{d,m}^\e ,x,t,y,s \right) -  F \left( Y_{d,m}^\e,x,t,y,s\right) \right),
\end{equation*}
where $Y_{d,m}^\e = Y_{d,m}^\e(x,t,y,s)$ represents the errors accumulated from the correction scheme over all the previous steps, including the error left from the initial layer correction at the $d$-th step.

It is noteworthy that $Y_{d,m}^\e(x,t,y,s)$ will turn out to be periodic in the space variable $y$ only, but not in the time variabler $s$. Instead, we shall discover that $Y_{d,m}^\e(x,t,y,s)$ converges exponentially fast to a space-time periodic term, as $s\ra\infty$. This indicates that $\tilde{F}_{d,m}^\e (P,x,t,y,s)$ is also not periodic in $s$, but admits a space-time periodic version. Recalling that the functional $F = F(P,x,t,y,s)$ was initially assumed to be periodic in both $y$ and $s$, we observe a change of the oscillating pattern of the operators appearing in the higher order correction scheme. We shall call such a change the nonlinear coupling effect, since it is observed only in nonlinear problems and it is produced by the correlation between the initial layer and interior correction. Let us remark that such a coupling effect has not been observed in any literature concerning convergence rate in the homogenization theory, including the authors' previous works \cite{KL1,KL2}. 

Finally, let us conclude the discussion with a remark on the homogenization problem \eqref{eq:ue-pde-nl}. The key difference in the homogenization process between \eqref{eq:ue-pde} and \eqref{eq:ue-pde-nl} is that the initial layer corrector of the latter problem may not be differentiable in the slow variables in general (see Example \ref{example:counter} for a counterexample), while the former produces smooth initial layer correctors. For this reason, the best one can hope for is convergence rate of order $\e$. We remark that $O(\e)$-estimates are obtained here for certain class of uniformly parabolic operators $F$ and initial datum $g$. We believe that the higher order convergence rate for \eqref{eq:ue-pde-nl} will be available under some additional conditions on $F$ and $g$. Classifying such $F$ and $g$ will require a thorough analysis on the limiting behavior of $\e^2 F(\e^{-2} P,x,t,y,s)$ as $\e\ra 0$, and we hope to come back to this problem in our future work. 


\subsection{Outline}\label{subsection:outline}

This paper is organized as follows. In Section \ref{section:assump}, we introduce notation and standing assumptions of our main result. In Section \ref{section:slow}, we establish the regularity theory for correctors, and in Section \ref{section:higher}, we construct higher order initial layer and interior correctors. Our main result, Theorem \ref{theorem:higher}, is proved in Section \ref{subsection:bootstrap}. Section \ref{section:further} is devoted to the higher order convergence rate in homogenization of \eqref{eq:ube-pde}, and the convergence rate in homogenization of \eqref{eq:ue-pde-nl}.


\section{Notation and Standing Assumptions}\label{section:assump}

Basic notations are summarized as follows.

\begin{itemize}
\item $n$, $\lambda$, $\Lambda$, $\alpha$, $K$, $T$: fixed parameters, with $n\geq 1$, $0<\lambda\leq\Lambda$, $0<\alpha\leq 1$, $K,T>0$.
\item $\lfloor a \rfloor$: largest integer smaller or equal to $a$. 
\item $\cS^n$:  space of all real symmetric matrices of order $n$, endowed with $(L^2,L^2)$-norm and $\{E_{ij}:1\leq i,j\leq n\}$  the set of standard basis matrices. 
\item $\T^n$: $n$-dimensional unit torus, $\R^n/\Z^n$; $\T = \T^1$. 
\item $x$, $t$: slow spatial, temporal variable; $x\in\R^n$, $t\in[0,T]$.
\item $y$, $s$: fast spatial, temporal variable; $y\in\T^n$, $s\in \T$ or $[0,\infty)$. 
\item $D_p^k F$: $k$-th order derivative of $F$ on $\cS^n$, and 
\begin{equation*}
D_p^k F(P) (Q_1,\cdots,Q_k) = \frac{\p^k F}{\p p_{i_1j_1} \cdots \p p_{i_kj_k}}(P) q_{i_1j_1}^1 \cdots q_{i_kj_k}^k,
\end{equation*}
where we use the summation convention for repeated indices. 
\item $C^{k,\alpha}(X)$: usual H\"older space on $X$. 
\item $E^{k,\alpha}(\T^n\times[0,\infty);\beta)$: space of all $f\in C^{k,\alpha}(\T^n\times[0,\infty))$ satisfying 
\begin{equation*}
\norm{f}_{E^{k,\alpha}(\T^n\times[0,\infty);\beta)} = \norm{f}_{C^{k,\alpha}(\T^n\times[0,\infty))} + \sup_{s>0}\left( e^{\beta s} \norm{f(\cdot,s)}_{C^{k,\alpha}(\T^n)} \right) < \infty. 
\end{equation*} 
\item $C^m(X;C^{k,\alpha}(Y))$: space of all $f : X\times Y \ra Z$ satisfying 
\begin{equation*}
\sup_{0\leq l\leq m}\sup_{\xi\in X}\norm{D_\xi^l f(\xi,\cdot)}_{C^{k,\alpha}(Y)}  < \infty.
\end{equation*} 
Here we can have $m = \infty$; possible candidates for $X = \R^n$ or $\R^n\times[0,T]$, $Y = \T^n$,  $\T^n\times\T$ or $\T^n\times[0,\infty)$, and $Z = \R$ or $\cS^n$. 
\item $C^m(X; E^{k,\alpha}(\T^n\times[0,\infty);\beta))$: defined similarly as $C^m(X; C^{k,\alpha}(Y))$ with $C^{k,\alpha}(Y)$ replaced by $E^{k,\alpha}(\T^n\times[0,\infty);\beta)$. 
\item $S(m,\bar\alpha;d,k)$: class of all $f\in C^\infty(X;C^{m,\bar\alpha}(Y))$ satisfying 
\begin{equation*}
\norm{ D_\xi^l f(\xi,\cdot)}_{C^{m,\bar\alpha}(Y)} \leq C_{m,\bar\alpha,d,k,l},\quad\text{for each }l\geq 0,
\end{equation*}
where $C_{m,\bar\alpha,d,k,l}>0$ depends only on $n$, $\lambda$, $\Lambda$, $\alpha$, $K$, $T$, $m$, $\bar\alpha$, $d$, $k$ and $l$; $S(m,\bar\alpha;k) = S(m,\bar\alpha;0,k)$, $S(d,k) = S(0,0;d,k)$, $S(k) = S(0,0;0,k)$. 
\item $E(m,\bar\alpha;d,k)$: class of all $f\in C^\infty(X; C^{m,\bar\alpha}(\T^n\times[0,\infty))) \cap  S(m,\bar\alpha;d,k)$ such that
\begin{equation*}
\norm{D_\xi^l f(\xi,\cdot)}_{E^{m,\bar\alpha}(\T^n\times[0,\infty);\beta_{d,k,l})} \leq C_{m,\bar\alpha,d,k,l},\quad\text{for each }l\geq 0, 
\end{equation*}
where $C_{m,\bar\alpha,d,k,l}>0$ depends on the same set of parameters as above, and $\beta_{d,k,l}$ depends only on $n$, $\lambda$, $\Lambda$, $d$, $k$ and $l$. 
\end{itemize}

We will use the parabolic terminology, such as $|(x,t)| = (|x|^2 + |t|)^{1/2}$. Readers may refer to \cite{W1}, \cite{W2}, \cite{W3} and \cite{CIL} for the existence and regularity theory for viscosity solutions of fully nonlinear parabolic equations. 

Standing assumptions on the interior operator $F:\cS^n\times\R^n\times[0,T]\times\T^n\times\T\ra \R$ and the initial data $g:\R^n\times\T^n\ra\R$ are given as follows. 
\begin{itemize}
\item (Uniformly Ellipticity) $F$ is uniformly elliptic on $\cS^n$:   
\begin{equation}\label{eq:ellip-F}
\lambda|Q| \leq F(P+Q,x,t,y,s) - F(P,x,t,y,s) \leq \Lambda|Q|,\quad\text{for any } P,Q\in\cS^n,Q\geq 0,
\end{equation}
\item (Concavity) $F$ is concave on $\cS^n$:
\begin{equation}\label{eq:concave-F}
\frac{1}{2} F(P,x,t,y,s) + \frac{1}{2} F (Q,x,t,y,s) \leq F\left(\frac{P+Q}{2},x,t,y,s\right),\quad \text{for any }P,Q\in\cS^n. 
\end{equation}
\item (Regularity) $F \in C^\infty(\cS^n\times\R^n\times[0,T];C^\alpha(\T^n\times\T))$, $g\in C^\infty(\R^n ; C^{2,\alpha}(\T^n))$ with
\begin{equation}\label{eq:Ck-Ca-F}
\norm{ D_p^k D_x^m \p_t^l F(P,x,t,\cdot,\cdot) }_{C^\alpha(\T^n\times\T)} \leq K|P|^{(1-k)_+},\quad \text{for any }k,m,l\geq 0,
\end{equation}
and 
\begin{equation}\label{eq:Ck-C2a-g}
\norm{ D_x^k g(x,\cdot) }_{C^{2,\alpha}(\T^n)}  \leq K,\quad\text{for any }k\geq 0.
\end{equation}
\end{itemize}
Note that the periodicity of $F$ and $g$ is implied in their domains. 


\section{Regularity Theory in Slow Variables}\label{section:slow}

Let us establish the regularity theory in slow variables, $(x,t)$, regarding viscosity solutions to uniformly parabolic problems. 

\subsection{Spatially Periodic Cauchy Problem}\label{subsection:decay}

Let $f  \in C^\infty(\R^n\times[0,T];E^\alpha(\T^n\times[0,\infty);\beta))$, with some $\beta>0$, such that 
\begin{equation}\label{eq:decay-Ck-Ca-f}
\begin{split}
\norm{D_x^k \p_t^l f(x,t,\cdot,\cdot)}_{E^\alpha(\T^n\times[0,\infty);\beta)} \leq K,\quad\text{for any }k, l\geq 0,
\end{split}
\end{equation}
with $\alpha$, $K$ as in \eqref{eq:Ck-Ca-F}. 

For each $(x,t)\in\R^n\times[0,T]$, let us consider the following spatially periodic and uniformly parabolic Cauchy problem,
\begin{equation}\label{eq:v-pde}
\begin{cases}
v_s = F(D_y^2 v,x,t,y,s) + f(x,t,y,s)& \text{in }\T^n\times(0,\infty),\\
v(x,t,y,0) = g(x,y) & \text{on }\T^n.
\end{cases}
\end{equation}
By the standard existence theory \cite{CIL}, we know that there exists a unique viscosity solution $v(x,t,\cdot,\cdot)\in BUC(\T^n\times[0,\infty))$ to \eqref{eq:v-pde}. Due to the periodicity of $F$, $f$ and $g$. 

We shall begin with an easy observation that the spatial oscillation of $v(x,t,y,s)$ in $y$ decays exponentially fast as $s\ra\infty$. The exponential rate will turn out to be independent of $(x,t)$. 

\begin{lemma}\label{lemma:decay} For each $(x,t)\in\R^n\times[0,T]$, there exists a unique $\gamma\in\R$ such that 
\begin{equation}\label{eq:decay}
e^{\beta_0 s}| v(x,t,y,s) - \gamma | \leq C,
\end{equation}
for any $(x,t,y,s)\in\R^n\times[0,T]\times\T^n\times[0,\infty)$, where $0<\beta_0<\beta$ depends only on $n$, $\lambda$, $\Lambda$ and $\beta$, and $C>0$ depend only on $n$, $\lambda$, $\Lambda$, $\beta$, $\beta_0$ and $K$. 
\end{lemma}

\begin{proof} We shall identify $\T^n$ by $[0,1]^n$ in $\R^n$, and extend $v(x,t,\cdot,\cdot)$ as a function on $\R^n\times[0,\infty)$ through its periodicity. Since $(x,t)$ will be fixed throughout the proof, let us write $v=v(y,s)$, $F=F(M,y,s)$, $f=f(y,s)$ and $g=g(y)$ for notational convenience. By $S(s)$, $I(s)$ and $O(s)$ let us denote $\sup_{\R^n} v(\cdot,s)$, $\inf_{\R^n} v(\cdot,s)$ and respectively $\osc_{\R^n} v(\cdot,s)$. Also write $Y = (0,1)^n$ and $2Y=(0,2)^n$. By the spatial periodicity of $v$, we have $S(s) = \sup_{2Y} v(\cdot,s) = \sup_Y v(\cdot,s)$, and similar identities for $I(s)$ and $O(s)$ as well.

Fix $s_0\geq 0$. Then 
\begin{equation}
\p_s \left( S(s_0) + \frac{K}{\beta} (e^{-\beta s_0} -e^{-\beta s})\right) = K e^{-\beta s} \geq f(y,s),
\end{equation}
for any $y\in\R^n$ and $s\geq s_0$, due to \eqref{eq:decay-Ck-Ca-f}. Since we have $F(0,y,s) = 0$ for all $(y,s)$, we deduce that $S(s_0) + \frac{K}{\beta} (e^{-\beta s_0} - e^{-\beta s})$ is a supersolution to \eqref{eq:v-pde} in $\R^n\times[s_0,\infty)$. Similarly, one can observe that $I(s_0) - \frac{K}{\beta} (e^{-\beta s_0} - e^{-\beta s})$ is a subsolution to \eqref{eq:v-pde} in $\R^n\times[s_0,\infty)$. Thus, by the comparison principle \cite{CIL} for viscosity solutions, we deduce that 
\begin{equation}\label{eq:M-m}
I(s_0) - \frac{K}{\beta}(e^{-\beta s_0} - e^{-\beta s}) \leq v(y,s) \leq S(s_0) + \frac{K}{\beta} (e^{-\beta s_0}- e^{-\beta s}), 
\end{equation}
for any $y\in\R^n$ and $s\geq s_0$. 

Now for each nonnegative integer $k$, let us define 
\begin{equation}
v_k(y,s) = v(y,s+k) - I(k) + \frac{K}{\beta} e^{-\beta k},\quad y\in\R^n,s\geq 0.
\end{equation}
From \eqref{eq:M-m} with $s$ and $s_0$ replaced by $s+k$ and $k$ respectively, we deduce that 
\begin{equation}
v_k(y,s) \geq 0,\quad y\in\R^n,s\geq 0.
\end{equation}
On the other hand, we see that $v_k$ is a (spatially periodic) viscosity solution to 
\begin{equation}
\p_s v_k = F(D_y^2 v_k, y,s+k) + f(y,s+k) \quad\text{in } 2Y\times(0,1).
\end{equation}
Therefore, we may apply the Harnack inequality in $\bar{Y}\times [\frac{1}{2},1]$ and deduce from the spatial periodicity of $v_k$ that
\begin{equation}
S\left(k+\frac{1}{2}\right) - I(k) + \frac{K}{\beta} e^{-\beta k} \leq c_1\left(I(k+1) - I(k) + \frac{K}{\beta} e^{-\beta k}\right),
\end{equation}
where $c_1$ depends only on $n$, $\lambda$ and $\Lambda$. Utilizing \eqref{eq:M-m} with $s_0= k+\frac{1}{2}$ and $s=k+1$, we obtain that $S(k+1) \leq S(k+\frac{1}{2}) + \frac{K}{\beta} e^{-\beta k}$. Combining these two inequalities, we arrive at
\begin{equation}\label{eq:Mm}
S(k+1) - I(k) \leq c_1 \left(I(k+1) - I(k) + \frac{K}{\beta} e^{-\beta k}\right).
\end{equation}

Now we define 
\begin{equation}
w_k(y,s) = S(k) + \frac{K}{\beta} e^{-\beta k} - v(y,s+k),\quad y\in\R^n,s\geq 0.
\end{equation}
Then by \eqref{eq:M-m} and \eqref{eq:v-pde}, $w_k$ is a spatially periodic nonnegative viscosity solution to 
\begin{equation}
\p_s w_k = - F( - D_y^2 w_k, y, s+k) - f(y,s+k) \quad \text{in }2Y\times(0,1).
\end{equation}
Notice that the operator $-F(-M,y,s)$ satisfies the same ellipticity condition \eqref{eq:ellip-F}. Hence, we may invoke a similar argument as above and prove that 
\begin{equation}\label{eq:mM}
S(k) - I(k+1) \leq c_1 \left(S(k) - S(k+1) + \frac{K}{\beta} e^{-\beta k}\right).
\end{equation}
Notice that the constant $c_1$ here is the same as that in \eqref{eq:Mm}. 

By \eqref{eq:Mm} and \eqref{eq:mM}, we have
\begin{equation}\label{eq:O}
O(k+1) \leq \frac{c_1-1}{c_1+1} O(k) + \frac{2c_1K}{(c_1+1)\nu} e^{-\beta k}. 
\end{equation}
Iterating \eqref{eq:O} with respect to $k$ and using $O(0) = \osc_{\R^n} g \leq 2K$, we arrive at 
\begin{equation}\label{eq:osc-decay}
e^{\beta_0 s} O(s) \leq c_2K\quad\text{with }0<\beta_0 < \min\left(\beta,\log\frac{c_1+1}{c_1-1}\right),
\end{equation}
where $c_2>0$ is another constant depending only on $n$, $\lambda$, $\Lambda$, $\beta$ and $\beta_0$.

The estimate \eqref{eq:osc-decay} implies that $O(s) \ra 0$ as $s\ra\infty$. On the other hand, we know from \eqref{eq:M-m} that both $S(s)$ and $I(s)$ converge as $s\ra\infty$. Combining these two observations, we deduce that $S(s)$ and $I(s)$ converge to the same limit, which we shall denote by $\gamma$. Then \eqref{eq:decay} follows immediately from \eqref{eq:osc-decay}. 
\end{proof}

\begin{remark}\label{remark:decay} The proof of Lemma \ref{lemma:decay} does not involve the periodicity of $F$ in $s$.
\end{remark}

By Lemma \ref{lemma:decay}, we are able to define $\bar{v}:\R^n\times[0,T]\ra\R$ by 
\begin{equation}\label{eq:vb}
\bar{v}(x,t) = \lim_{s\ra\infty} v(x,t,0,s).
\end{equation}
The limit value in the right hand side of \eqref{eq:vb} is precisely the unique constant $\gamma$ in the statement of Lemma \ref{lemma:decay}. With $\bar{v}$ at hand, \eqref{eq:decay} reads 
\begin{equation}\label{eq:decay-re}
e^{\beta_0 s} | v(x,t,y,s) - \bar{v}(x,t) | \leq C,
\end{equation}
for any $(x,t,y,s)\in\R^n\times[0,T]\times\T^n\times[0,\infty)$. 

One may notice that the proof of Lemma \ref{lemma:decay} does not involve the assumptions on the concavity of $F$ in $P$, the $C^\alpha$ regularity of $F$ and $f$ in $(y,s)$ and the $C^{2,\alpha}$ regularity of $g$ in $y$. Assuming these conditions additionally, we are allowed to use the interior and boundary $C^{2,\bar\alpha}$ estimates (the so-called Schauder theory) for viscosity solutions (with some $0<\bar\alpha\leq\alpha$). As a result, we improve the estimate \eqref{eq:decay-re} in terms of $C^{2,\bar\alpha}$ norm.

\begin{lemma}\label{lemma:decay-Linf-C2a} There exists $0<\bar\alpha<\alpha$, depending only on $n$, $\lambda$, $\Lambda$ and $\alpha$, such that $\bar{v} \in L^\infty(\R^n\times[0,T])$ and $v\in L^\infty(\R^n\times[0,T]; C^{2,\bar\alpha}(\T^n\times[0,\infty)))$ with 
\begin{equation}\label{eq:decay-Linf-C2a}
| \bar{v}(x,t)| + \norm{ v (x,t,\cdot,\cdot) - \bar{v}(x,t)}_{E^{2,\bar\alpha}(\T^n\times[0,\infty);\beta_0)} \leq C,
\end{equation}
for any $(x,t,s)\in\R^n\times[0,T]$, where $C>0$ depends only on $n$, $\lambda$, $\Lambda$, $\beta$, $\beta_0$ and $K$.  
\end{lemma}

\begin{proof} Let us fix $(x,t)\in\R^n\times[0,T]$ and simply write $F(P,y,s)$,  $f(y,s)$, $g(y)$, $v(y,s)$, and $\gamma$ for $F(P,x,t,y,s)$, $f(x,t,y,s)$, $g(x,y)$, $v(x,t,y,s)$ and, respectively, $\bar{v}(x,t)$. Let us extend $v$ to a function on $\R^n\times[0,\infty)$ as in the proof of Lemma \ref{lemma:decay}. Also denote by $Y$ and $2Y$ the cubes $(0,1)^n$ and respectively $(0,2)^n$. 

In view of \eqref{eq:v-pde}, the function $\tilde{v}(y,s) = v(y,s) - \gamma$ is a viscosity solution to
\begin{equation}\label{eq:vt-pde}
\begin{cases}
\tilde{v}_s = F( D_y^2 \tilde{v}, y,s) + f(y,s) &\text{in } 2Y\times (0,\infty),\\
\tilde{v}(y,0) = g(y) - \gamma & \text{on } 2Y.
\end{cases}
\end{equation}
Since $F$ is uniformly elliptic and concave in $P$, and since $F$ and $f$ are $C^\alpha$ while $g$ is $C^{2,\alpha}$ in $(y,s)$, we may apply the boundary $C^{2,\bar\alpha}$ estimate \cite{W2} to \eqref{eq:vt-pde} for some $0<\bar\alpha\leq \alpha$, depending only on $n$, $\lambda$, $\Lambda$ and $\alpha$. This yields that $\tilde{v} \in C^{2,\bar\alpha} (\bar{Y}\times[0,s_0])$ with 
\begin{equation}
\begin{split}
\norm{ \tilde{v} }_{C^{2,\bar\alpha}(\bar{Y}\times[0,s_0])} & \leq c_1\left( \norm{\tilde{v}}_{L^\infty(2Y\times [0,1))} + \norm{f}_{C^\alpha(2Y\times [0,1))} + \norm{g}_{C^{2,\alpha}(2Y)}\right),
\end{split}
\end{equation}
where $0<s_0\leq \frac{1}{2}$ and $c_1>0$ depend only on $n$, $\lambda$, $\Lambda$, $\alpha$ and $K$. Utilizing \eqref{eq:decay}, \eqref{eq:decay-Ck-Ca-f} and \eqref{eq:Ck-C2a-g} (with $m=0$), we derive that
\begin{equation}\label{eq:vt-C2a-bdry}
\norm{\tilde{v}}_{C^{2,\bar\alpha} (Y\times[0,s_0])} \leq c_2,
\end{equation}
where $c_2>0$ is determined only by $n$, $\lambda$, $\Lambda$, $\alpha$, $\beta$, $\beta_0$ and $K$.

Now let us fix a nonnegative integer $k$ and define 
\begin{equation}
\tilde{v}_k (y,s) = \tilde{v}(y,s+k) \quad (y\in\R^n,s\geq 0). 
\end{equation}
Then from \eqref{eq:vt-pde}, we know that $\tilde{v}_k$ solves 
\begin{equation}
\p_s \tilde{v}_k = F(D_y^2 \tilde{v}_k,y,s+k) + f(y,s+k)\quad\text{in } 2Y\times(0,2).
\end{equation}
Hence, it follows from the interior $C^{2,\bar\alpha}$ estimate (with $\bar\alpha$ being the same as that in \eqref{eq:vt-C2a-bdry}) that $\tilde{v}_k\in C^{2,\bar\alpha} ( \bar{Y}\times[s_0,s_0+1])$ with 
\begin{equation}
\norm{\tilde{v}_k}_{C^{2,\bar\alpha}(\bar{Y}\times [s_0,s_0+1])} \leq c_3\left( \norm{\tilde{v}_k}_{L^\infty(2Y\times (0,2))} + \norm{f}_{C^\alpha(2Y\times (0,2))}\right),
\end{equation}
where $c_3>0$ depends only on $n$, $\lambda$, $\Lambda$, $\alpha$ and $K$. Utilizing \eqref{eq:decay} and \eqref{eq:decay-Ck-Ca-f} (with $m=0$), we deduce that
\begin{equation}\label{eq:vt-C2a-int}
\norm{\tilde{v}_k}_{C^{2,\bar\alpha}(\bar{Y}\times[s_0,s_0+1])} \leq c_4e^{-\beta_0 k},
\end{equation}
where $c_4>0$ is determined only by $n$, $\lambda$, $\Lambda$, $\alpha$, $\beta$, $\beta_0$ and $K$. 

Iterating \eqref{eq:vt-C2a-int} with respect to $k$ and utilizing \eqref{eq:vt-C2a-bdry} for the initial case of this iteration argument, we arrive at \eqref{eq:decay-Linf-C2a}.
\end{proof}

Let us remark that Lemma \ref{lemma:decay-Linf-C2a} yields the compactness (in $(y,s)$) of the sequences $\{v(x_i,t_i,y,s)\}_{i=1}^\infty$ and $\{\tilde{v}(x_i,t_i,y,s)\}_{i=1}^\infty$ when $(x_i,t_i)\ra (x,t)$. By the stability theory \cite{CIL} of viscosity solutions, we obtain that $v$ and $\tilde{v}$ are continuous in $(x,t)$, stated as below. Let us also point out that the following lemma is a version of continuous dependence estimates, and we refer to \cite{JK}, \cite{CM} and other literature for more discussions in this regard. 

\begin{lemma}\label{lemma:decay-C0-C2a}
Let $\bar\alpha$ be the H\"{o}lder exponent chosen in Lemma \ref{lemma:decay-Linf-C2a}. Then $\bar{v} \in C(\R^n\times[0,T])$ and $v \in C (\R^n\times[0,T];C_{loc}^{2,\hat\alpha}(\T^n\times[0,\infty)))$ for any $0<\hat\alpha<\bar\alpha$. 
\end{lemma}

\begin{proof} As in the proof of Lemma \ref{lemma:decay-Linf-C2a}, we will fix $(x,t)\in\R^n\times[0,T]$ and continue with using the simplified notation for $F$, $f$, $g$, $v$, $\gamma$ and $\tilde{v}$. Let us take any sequence $(x_i,t_i)\ra (x,t)$ as $i\ra\infty$. For notational convenience, let us write $F_i (P,y,s) = F(P,x_i,t_i,y,s)$, $f_i(y,s) = f(x_i,t_i,y,s)$, $g_i(y) = g(x_i,y)$, $v_i (y,s) = v(x_i,t_i,y,s)$, $\gamma_i = \bar{v}(x_i,t_i)$ and $\tilde{v}_i(y,s) = v_i(y,s) - \gamma_i$. By $C$ we denote a positive constant that depends only on $n$, $\lambda$, $\Lambda$, $\alpha$, $\beta$, $\beta_0$ and $K$, and will let it vary from one line to another.

We prove $v_i\ra v$ first. By \eqref{eq:decay-Linf-C2a} we have 
\begin{equation*}
\norm{v_i}_{C^{2,\bar\alpha} (\T^n\times[0,\infty))} \leq C,
\end{equation*}
for any $i=1,2,\cdots$. Hence, we know from the Arzela-Ascoli theorem that for any subsequence $\{w_j\}_{j=1}^\infty$ of $\{v_i\}_{i=1}^\infty$, there exist a further subsequence $\{w_{j_k}\}_{k=1}^\infty$ and a certain function $w\in C^{2,\bar\alpha}(\T^n\times[0,\infty))$ such that $w_{j_k} \ra w$ in $C_{loc}^{2,\hat\alpha}(\T^n\times[0,\infty))$ as $k\ra\infty$, for any $0<\hat\alpha<\bar\alpha$. One may notice that $w_{j_k}$ solves 
\begin{equation}
\begin{cases}
\p_s w_{j_k} = F_{j_k} (D_y^2 w_{j_k},y,s) + f_{j_k} (y,s) & \text{in }\T^n\times(0,\infty),\\
w_{j_k} (y,0) = g_{j_k} (y) & \text{in }\T^n,
\end{cases}
\end{equation}
in the viscosity sense. Due to the regularity assumptions \eqref{eq:Ck-Ca-F}, \eqref{eq:Ck-C2a-g} and \eqref{eq:decay-Ck-Ca-f} on $F$, $g$ and respectively $f$, we know that $F_i \ra F$ uniformly on $\cS^n\times\T^n\times[0,\infty)$, $g_i\ra g$ uniformly on $\R^n$ and $f_i\ra f$ uniformly on $\T^n\times[0,\infty)$, as $i\ra\infty$. Hence, letting $k\ra\infty$, we observe from the stability theory \cite{CIL} that the limit function $w$ also solves 
\begin{equation}
\begin{cases}
w_s = F (D_y^2 w,y,s) + f(y,s) & \text{in }\T^n\times(0,\infty),\\
w (y,0) = g(y) & \text{on }\T^n,
\end{cases}
\end{equation}
in the viscosity sense. However, the above equation is identical with the equation \eqref{eq:v-pde}. Since $v$ is the unique solution to \eqref{eq:v-pde}, we deduce that $w=v$ on $\T^n\times[0,\infty)$. 

What we have proved so far is that for any subsequence of $\{v_i\}_{i=1}^\infty$, there exists a further subsequence which converges to $v$. Thus, $v_i \ra v$ as $i\ra\infty$ in $C_{loc}^{2,\hat\alpha}(\T^n\times[0,\infty))$ for any $0<\hat\alpha<\bar\alpha$. 

Now we are left with showing that $\gamma_i  \ra \gamma$. Due to \eqref{eq:decay-Linf-C2a}, we have 
\begin{equation}\label{eq:vti-decay-Ca}
e^{\beta_0 s}\norm{ \tilde{v}_i (\cdot,s) }_{C^{2,\bar\alpha}(\T^n)} \leq C,
\end{equation} 
for any $s\geq 0$, uniformly for all $i=1,2,\cdots$. Since $\tilde{v}_i(y,s) = v_i(y,s) - \gamma_i$ and $\tilde{v}(y,s) = v(y,s) - \gamma$, we deduce from \eqref{eq:vti-decay-Ca} and \eqref{eq:decay-Linf-C2a} that
\begin{equation}
| \gamma_i - \gamma | \leq 2Ce^{-\beta_0 s} + |v_i(0,s) - v(0,s)|.
\end{equation}
Given any $\delta>0$, we fix a sufficiently large $s_0$ such that $4Ce^{-\beta_0 s_0}\leq \delta$, and correspondingly choose $i_0$ such that $2|v_i(0,s_0) - v(0,s_0)| \leq \delta$ for all $i\geq i_0$. Then we have $|\gamma_i - \gamma|\leq \delta$ for all $i\geq i_0$, proving that $\gamma_i\ra \gamma$ as $i\ra\infty$. Thus, the proof is finished.
\end{proof}

By Lemma \ref{lemma:decay-C0-C2a}, we are ready to prove the differentiability of $v$ and $\bar{v}$ in the slow variables $(x,t)$, and an exponential decay estimate for the derivatives of $v-\bar{v}$. Here we use Lemma \ref{lemma:decay-C0-C2a} to obtain compactness (in $(y,s)$) of the difference quotients (in $(x,t)$) of $v$. Arguing similarly as in the proof of Lemma \ref{lemma:decay-C0-C2a}, we deduce that the difference quotients converge to a single limit, proving the differentiability of $v$. 

\begin{lemma}\label{lemma:decay-C1-C2a} Let $\bar\alpha$ be the H\"{o}lder exponent chosen in Lemma \ref{lemma:decay-Linf-C2a}. Then there exist $D_x \bar{v}(x,t)$ and $D_x v(x,t,\cdot,\cdot)\in C^{2,\bar\alpha} (\T^n\times[0,\infty))$ such that
\begin{equation*}
\begin{split}
| D_{x_k} \bar{v}(x,t) |  +  \norm{ D_{x_k} (v(x,t,\cdot,\cdot) - \bar{v}(x,t))}_{E^{2,\bar\alpha}(\T^n\times[0,\infty);\beta_1)} \leq C,
\end{split}
\end{equation*}
for any $(x,t)\in\R^n\times[0,T]$, where $0<\beta_1<\beta_0$ depends only on $n$, $\lambda$, $\Lambda$ and $\beta_0$, and $C>0$ depends only on $n$, $\lambda$, $\Lambda$, $\alpha$, $\beta$, $\beta_0$, $\beta_1$ and $K$. Moreover, we have $\bar{v} \in C^1(\R^n\times[0,T])$ and $v\in C^1(\R^n\times[0,T]; C_{loc}^{2,\hat\alpha} (\T^n\times[0,\infty)))$ for any $0<\hat\alpha<\bar\alpha$. 
\end{lemma}

\begin{remark}\label{remark:decay-C1-C2a} According to the parabolic terminology, $C^1$ regularity in $(x,t)$ only involves derivatives in $x$. For more details, see Section \ref{section:assump}.
\end{remark}

\begin{proof}[Proof of Lemma \ref{lemma:decay-C1-C2a}] Throughout this proof, let us write by $C$ a positive constant depending only on $n$, $\lambda$, $\Lambda$, $\alpha$, $\beta$ and $K$, and allow it to vary from one line to another. Fix $(x,t)\in\R^n\times[0,T]$ and $1\leq k\leq n$. We shall omit the dependence on $t$ for notational convenience. Let us define
\begin{align*}
A_\sigma (y,s) & = \int_0^1 D_p F ( \rho D_y^2 v(x+\sigma e_k,y,s) + (1-\rho) D_y^2 v(x,y,s),x,y,s) d\rho,\\
\Psi_\sigma (y,s) & = \frac{F(D_y^2 v(x+\sigma e_k,y,s), x+\sigma e_k, y,s) - F(D_y^2 v(x+\sigma e_k,y,s),x,y,s)}{\sigma} \\
&\quad + \frac{f(x+\sigma e_k,y,s) - f(x,y,s)}{\sigma},\\
G_\sigma (y) & = \frac{g(x+\sigma e_k,y) - g(x,y)}{\sigma},
\end{align*}
for $(y,s)\in\T^n\times[0,\infty)$, and nonzero $\sigma\in\R$. 

Clearly, $A_\sigma$, $\Psi_\sigma$ and $G_\sigma$ are periodic in $y$. The ellipticity of $A_\sigma$ follows immediately from \eqref{eq:ellip-F}. Indeed, $A_\sigma$ satisfies
\begin{equation}\label{eq:As-ellip}
\lambda |Q| \leq \tr( A_\sigma (y,s) Q) \leq \Lambda |Q|\quad (Q\in\cS^n,Q\geq 0),
\end{equation}
for any $(y,s)\in\T^n\times[0,\infty)$. It should be remarked that the lower and the upper ellipticity bounds of $A_\sigma$ are not only independent of $\sigma$ but also the same as those of $F$. 

By \eqref{eq:Ck-Ca-F} and \eqref{eq:decay-Linf-C2a}, we know that $A_\sigma\in C^{\bar\alpha}(\T^n\times[0,\infty))$ and 
\begin{equation}\label{eq:As-Ca}
\norm{ A_\sigma}_{C^{\bar\alpha}(\T^n\times[0,\infty))} \leq C.
\end{equation}
Let us remark here that we need Lipschitz regularity of $D_pF$ in $P$ in order to have \eqref{eq:As-Ca}. 

Similarly, we may deduce from 
\eqref{eq:Ck-Ca-F}, \eqref{eq:decay-Ck-Ca-f} and \eqref{eq:decay-Linf-C2a} that $\Psi_\sigma\in C^{\bar\alpha} (\T^n\times[0,\infty))$ satisfies
\begin{equation}\label{eq:Psis-decay-Ca}
\norm{\Psi_\sigma}_{E^{\bar\alpha}(\T^n\times[0,\infty);\beta_0)}  \leq C.
\end{equation}
On the other hand, it follows directly from \eqref{eq:Ck-C2a-g} that  $G\in C^{2,\alpha}(\R^n)$ and 
\begin{equation}\label{eq:Gs-C2a}
\norm{ G_\sigma }_{C^{2,\alpha}(\T^n)}  \leq K.
\end{equation}

Now we define 
\begin{equation}
V_\sigma (y,s)  = \frac{v(x+\sigma e_k,y,s) - v(x,y,s)}{\sigma}\quad\text{and}\quad \Gamma_\sigma = \frac{ \bar{v}(x+ \sigma e_k) - \bar{v}(x) }{\sigma},
\end{equation}
for $(y,s)\in\T^n\times[0,\infty)$ and nonzero $\sigma\in\R$. Linearizing the equation \eqref{eq:v-pde}, we see that $V_\sigma$ is a viscosity solution to 
\begin{equation}\label{eq:Vs-pde}
\begin{cases}
\p_s V_\sigma = \tr( A_\sigma(y,s) D_y^2 V_\sigma) + \Psi_\sigma(y,s) & \text{in }\T^n\times(0,\infty),\\
V_\sigma(y,0) = G_\sigma (y) & \text{on }\T^n. 
\end{cases}
\end{equation}

Owing to \eqref{eq:As-Ca} - \eqref{eq:Gs-C2a}, we observe that the equation \eqref{eq:Vs-pde} belongs to the same class of \eqref{eq:v-pde}. Hence, Lemma \ref{lemma:decay-Linf-C2a} is applicable to the problem \eqref{eq:Vs-pde}. In particular, the exponent $\beta$ in the statement of Lemma \ref{lemma:decay-Linf-C2a} is replaced here by $\beta_0$. Thus, we obtain some $0<\beta_1<\beta_0$, depending only on $n$, $\lambda$, $\Lambda$ and $\beta_0$, such that 
\begin{equation}\label{eq:Vs-decay-Ca}
| \Gamma_\sigma|  + \norm{V_\sigma}_{E^{2,\bar\alpha}(\T^n\times[0,\infty);\beta_1)} \leq C.
\end{equation}

Now we invoke the compactness argument used in the proof of Lemma \ref{lemma:decay-C0-C2a}. Choose any sequence $\sigma_i\ra 0$ as $i\ra\infty$. Then by \eqref{eq:Vs-decay-Ca}, there exist a subsequence $\{\tau_j\}_{j=1}^\infty$ of $\{\sigma_i\}_{i=1}^\infty$ and a function $V \in C^{2,\bar\alpha}(\T^n\times[0,\infty))$ such that $V_{\tau_j} \ra V$ in $C_{loc}^{2,\hat\alpha} (\T^n\times[0,\infty))$ as $j\ra\infty$, for any $0<\hat\alpha<\bar\alpha$. 

On the other hand, from the regularity assumptions on $F$ and $f$ (\eqref{eq:Ck-Ca-F} and \eqref{eq:decay-Ck-Ca-f} respectively) and the continuity of $D_y^2 v(x,t,y,s)$ in $(x,t)$ (Lemma \ref{lemma:decay-C0-C2a}), we deduce that $A_\sigma \ra A$ and $\Psi_\sigma \ra \Psi$ locally uniformly in $\T^n\times[0,\infty)$ as $\sigma \ra 0$, where
\begin{align*}
A(y,s) &= D_p F( D_y^2 v(x,y,s),x,y,s),\\
\Psi(y,s) & = D_{x_k} F(D_y^2 v (x,y,s),x,y,s) + D_{x_k} f(x,y,s).
\end{align*}
It also follows from the regularity assumption \eqref{eq:Ck-C2a-g} on $g$ that $G_\sigma \ra G$ uniformly in $\R^n$ with 
\begin{equation}
G(y) = D_{x_k} g(x,y).
\end{equation} 

Hence, it follows from the stability of viscosity solutions (see \cite{CIL} for the details) that the limit function $V$ of $V_{\tau_j}$ is a viscosity solution to 
\begin{equation}\label{eq:V-pde}
\begin{cases}
V_s = \tr( A(y,s) D_y^2 V) + \Psi(y,s) & \text{in }\T^n\times(0,\infty),\\
V(y,0) = G (y) & \text{on }\T^n. 
\end{cases}
\end{equation}
However, $A$, $G$ and $\Psi$ also satisfy \eqref{eq:As-ellip}, \eqref{eq:Gs-C2a} and respectively \eqref{eq:Psis-decay-Ca}. Thus, \eqref{eq:V-pde} belongs to the same class of \eqref{eq:v-pde}, which implies that $V$ is the unique (spatially periodic) viscosity solution to \eqref{eq:V-pde}. This shows that $V_\sigma \ra V$ in $C_{loc}^{2,\hat\alpha}(\T^n\times[0,\infty))$ as $\sigma\ra 0$, for any $0<\hat\alpha<\bar\alpha$. In other words, 
\begin{equation}
V (y,s) = D_{x_k} v(x,y,s).
\end{equation}  

Equipped with the uniform estimate \eqref{eq:Vs-decay-Ca} and the observation that $V_\sigma \ra V$, we may also prove that $\Gamma_\sigma \ra \Gamma$ as $\sigma\ra 0$, for some $\Gamma\in\R$. Since this part repeats the argument used in the end of the proof of Lemma \ref{lemma:decay-C0-C2a}, we skip the details. Let us remark that
\begin{equation}
\Gamma = D_{x_k} \bar{v}(x).
\end{equation} 

The second assertion of Lemma \ref{lemma:decay-C1-C2a} can be justified by following the proof of Lemma \ref{lemma:decay-C0-C2a} regarding \eqref{eq:V-pde}. To avoid the redundancy of the argument, we omit the details.
\end{proof}

From the proof of Lemma \ref{lemma:decay-C1-C2a}, we observe that the regularity of $v$ and $\bar{v}$ in $(x,t)$ can be improved in a systematic way. Induction on the order of the derivatives (in $(x,t)$) of $v$ and $\bar{v}$ leads us to the following proposition.

\begin{proposition}\label{proposition:decay-Ck-C2a} Under the assumptions \eqref{eq:ellip-F} - \eqref{eq:Ck-C2a-g} and \eqref{eq:decay-Ck-Ca-f} on $F$, $g$ and $f$, $v\in C^\infty (\R^n\times[0,T]; C^{2,\bar\alpha}(\T^n\times[0,\infty)))$ and $\bar{v}\in C^\infty(\R^n\times[0,T])$ with
\begin{equation}\label{eq:decay-Ck-C2a}
\begin{split}
\sum_{k + 2l = m} \left[ \left| D_x^k \p_t^l \bar{v}(x,t) \right| + \norm{ D_x^k \p_t^l ( v(x,t,\cdot,\cdot) - \bar{v}(x,t) ) }_{E^{2,\bar\alpha}(\T^n\times[0,\infty);\beta_m)} \right] \leq C_m,
\end{split}
\end{equation}
for all $(x,t)\in\R^n\times[0,T]$ and any $m\geq 0$, where $0<\beta_m<\beta$ depends only on $n$, $\lambda$, $\Lambda$, $m$ and $\beta$, and $C_m>0$ depends only on $n$, $\lambda$, $\Lambda$, $\alpha$, $\beta$, $m$ and $K$.
\end{proposition}

\begin{remark}\label{remark:decay-Ck-C2a} As pointed out in Remark \ref{remark:decay}, the proof of this proposition does not use the periodicity of $F$ in $s$. Moreover, $0<\beta_m<\cdots<\beta_0<\beta$ for any $m\geq 1$ and $C_m$ depends on the choice of $\beta_0,\cdots,\beta_m$. 
\end{remark}

\begin{proof}[Proof of Proposition \ref{proposition:decay-Ck-C2a}]

The proof of this proposition repeats that of Lemma \ref{lemma:decay-C1-C2a}. One may notice that although the statement of this lemma only involves the derivatives in $x$, the proof works equally well for the derivatives in $t$. Here we will only provide the sketch of the proof, and leave out the details to avoid redundancy.

Let $V_k$ and $\bar{V}_k$ be the $k$-th order derivative (in $(x,t)$) of $v$ and respectively $\bar{v}$. Let ($P_k$) be the equation which $V_k$ solves, and suppose (as the induction hypothesis) that the coefficient $A_k$, the source term $\Psi_k$ and the initial data $G_k$ of ($P_k$) belong to the same class of \eqref{eq:v-pde}. We know that this hypothesis is satisfied when $k=1$, since in that case the equation ($P_k$) is precisely \eqref{eq:V-pde}. By the induction hypothesis, Lemma \ref{lemma:decay-C0-C2a} is applicable, which gives us higher regularity of $V_k$ in the fast variables. 

Now let $\{V_{k,\sigma}\}_{\sigma\neq 0}$ be the sequence of difference quotients of $V_k$ (in $(x,t)$). To avoid confusion, let us denote by ($P_{k,\sigma}$) the equation for $V_{k,\sigma}$. Let us also denote by $A_{k,\sigma}$, $\Psi_{k,\sigma}$ and $G_{k,\sigma}$ the coefficient, the source term and respectively the initial data of ($P_{k,\sigma}$). 

Following the proof of Lemma \ref{lemma:decay-C1-C2a}, we may observe that ($P_{k,\sigma}$) is obtained by linearizing ($P_k$). Utilizing the structure conditions of $F$, $f$ and $g$, one may deduce that ($P_{k,\sigma}$) belongs to the same class of ($P_k$) with the structure conditions for ($P_{k,\sigma}$) being independent of $\sigma$. Moreover, one may observe from the regularity assumptions on $F$ and $f$ that $A_{k,\sigma}$ and $\Psi_{k,\sigma}$ converge to some $A_{k+1}$ and $\Psi_{k+1}$, respectively, as $\sigma\ra 0$ locally uniformly in the underlying domain of $(y,s)$. Here one needs to use the continuity of $D_y^2 V_k$ in $(x,t)$ that will be given in the induction hypotheses. On the other hand, $G_{k,\sigma}$ will converge to some $G_{k+1}$ uniformly in $y$, due to the regularity assumption on $G$. 

Hence, the stability theory of viscosity solutions will ensure that any limit of $V_{k,\sigma}$ is a viscosity solution to the problem ($P_{k+1}$) having $A_{k+1}$, $\Psi_{k+1}$ and $G_{k+1}$ as the coefficient, the source term and, respectively, the initial data. Then the uniqueness of (viscosity) solutions to ($P_{k+1}$) will lead us to the observation that $V_{k,\sigma}$ converges to a single limit function, say $V_{k+1}$. In other words, $V_k$ is differentiable (in $(x,t)$) and the corresponding derivative is $V_{k+1}$. Utilizing this fact, one may also prove that $\bar{V}_k$ is differentiable with the derivative being $\bar{V}_{k+1}$. 

We observe that Lemma \ref{lemma:decay-C0-C2a} provides us the desired estimate for $V_k$ and $\bar{V}_k$, while Lemma \ref{lemma:decay-C1-C2a} yields that for $V_{k+1}$ and $\bar{V}_{k+1}$. The rest of the proof can now be finished by an induction argument. 
\end{proof}


\subsection{Cell Problem}\label{subsection:cell}

Due to the uniform ellipticity and the periodicity of $F$, we know from the classical work \cite{E2} that there is a functional $\bar{F}:\cS^n\times\R^n\times[0,T]\ra \R$ such that for each $(P,x,t)\in\cS^n\times\R^n\times[0,T]$, the following equation,
\begin{equation}\label{eq:cell}
w_s = F( D_y^2 w + P, x,t, y,s) - \bar{F}(P,x,t) \quad \text{in }\T^n\times\T,
\end{equation}
has a periodic viscosity solution $w \in C(\T^n\times\T)$. We also know that $\bar{F}$ is uniformly elliptic with the same ellipticity constants of $F$, and it is concave in the Hessian variable $P$.  

Moreover, periodic viscosity solutions to \eqref{eq:cell} are unique up to an additive constant, if any. This also allows us to define another functional $w:\cS^n\times\R^n\times[0,T]\times\T^n\times\T\ra\R$ such that $w(P,x,t,\cdot,\cdot)$ is the unique viscosity solution to \eqref{eq:cell} which also satisfies
\begin{equation*}
w(P,x,t,0,0) = 0.
\end{equation*}

We shall now study the regularity of $\bar{F}$ and $w$ in $(P,x,t)$, which follows closely to the authors' previous work \cite{KL1}. We begin by improving the regularity of $w$ in the fast variables $(y,s)$, based on the interior $C^{2,\alpha}$ estimates \cite{W2} for viscosity solutions to concave equations. We leave the proof to the reader, as it is straightforward from the classical regularity result, and the property of the cell problem. 

\begin{lemma}\label{lemma:cell-C0-C2a} There exists $0<\bar\alpha\leq\alpha$ depending only on $n$, $\lambda$, $\Lambda$ and $\alpha$ such that $w(P,x,t,\cdot,\cdot)\in C^{2,\bar\alpha}(\T^n\times\T)$ with
\begin{equation*}
\norm{ w(P,x,t,\cdot,\cdot) }_{C^{2,\bar\alpha}(\T^n\times\T)} \leq C|P|,
\end{equation*}
for each $(P,x,t)\in\cS^n\times\R^n\times[0,T]$, where $C>0$ depends only on $n$, $\lambda$, $\Lambda$, $\alpha$ and $K$. Moreover, $w\in C(\cS^n\times\R^n\times[0,T];C^{2,\hat\alpha}(\T^n\times\T))$ for any $0<\hat\alpha<\bar\alpha$. 
\end{lemma}

With the above lemma at hand, we can proceed with the proof of (continuous) differentiability of $\bar{F}$ and $w$ in $(P,x,t)$. The proof is also similar to that of Lemma \ref{lemma:decay-C1-C2a}.

\begin{lemma}\label{lemma:cell-C1-Ca} Let $\bar\alpha$ be the H\"{o}lder exponent chosen in Lemma \ref{lemma:cell-C0-C2a}. Then there exist $D_p^k D_x^l \bar{F} (P,x,t)$ and $D_p^k D_x^l w(P,x,t,\cdot,\cdot) \in C^{2,\bar\alpha}(\T^n\times\T)$, for any $k,l\geq 0$ with $k+l=1$.  such that 
\begin{equation*}
\left| D_p^k D_x^l \bar{F} (P,x,t) \right| + \norm{ D_p^k D_x^l w(P,x,t,\cdot,\cdot) }_{C^{2,\bar\alpha}(\T^n\times\T)} \leq C|P|^{1-k},
\end{equation*} 
for any $(P,x,t)\in\cS^n\times\R^n\times[0,T]$, where $C>0$ depends only on $n$, $\lambda$, $\Lambda$, $\alpha$ and $K$. Moreover, we have $\bar{F} \in C^1(\cS^n\times\R^n\times[0,T])$ and $w\in C^1(\cS^n\times\R^n\times[0,T];C^{2,\hat\alpha}(\T^n\times\T))$ for any $0<\hat\alpha<\bar\alpha$.
\end{lemma}

\begin{remark}\label{remark:cell-C1-Ca} As pointed out in Remark \ref{remark:decay-C1-C2a}, $C^1$ regularity in $(P,x,t)$ does not involve that in $t$, according to the parabolic terminology. 
\end{remark}

\begin{proof}[Proof of Lemma \ref{lemma:cell-C1-Ca}] In this proof, we use $C$ to denote a positive constant that depends only on $n$, $\lambda$, $\Lambda$, $\alpha$ and $K$, and allow it to vary from one line to another. We shall prove this lemma for the derivatives in $P$ only, since the same argument applies to the proof for the derivatives in $x$. Fix $(P,x,t)\in\cS^n\times\R^n\times[0,T]$ and $1\leq i,j\leq n$. Recall from Section \ref{section:assump} that by $E_{ij}$ we denote the $(i,j)$-th standard basis matrix in $\cS^n$. For notational convenience, we shall skip the dependence of $F$, $w$ and $\bar{F}$ on $(x,t)$. Define 
\begin{align*}
A_\sigma (y,s) &= \int_0^1 D_p F( \rho D_y^2 w(P + \sigma E_{ij}, y,s) + (1-\rho) D_y^2 w(P,y,s) + \rho\sigma E_{ij}, y,s) d\rho,\\
W_\sigma (y,s) &= \frac{w(P + \sigma E_{ij},y,s) - w(P,y,s)}{\sigma} \quad \text{and}\quad \Gamma_\sigma = \frac{\bar{F}(P + \sigma E_{ij}) - \bar{F}(P)}{\sigma},
\end{align*}
for $(y,s)\in\T^n\times\T$. By linearization, we deduce that $W_\sigma$ is a (viscosity) solution to 
\begin{equation}\label{eq:Ws-pde}
\p_s W_\sigma = \tr( A_\sigma(y,s) (D_y^2 W_\sigma + E_{ij})) - \Gamma_\sigma\quad\text{in }\T^n\times\T. 
\end{equation}

Clearly, $A_\sigma$ is periodic on $\T^n\times\T$. More importantly, $A_\sigma$ is uniformly elliptic in the sense of \eqref{eq:As-ellip} and H\"{o}lder continuous with the uniform estimate \eqref{eq:As-Ca}. It should be stressed that the lower and upper ellipticity bounds for of $A_\sigma$ are given by $\lambda$ and, respectively, $\Lambda$ and are independent of $\sigma$. Hence, \eqref{eq:Ws-pde} belongs to the same class of \eqref{eq:cell}. As a result, Lemma \ref{lemma:cell-C0-C2a} is applicable to \eqref{eq:Ws-pde}. This yields that $W_\sigma \in C^{2,\bar\alpha} (\T^n\times\T)$ and 
\begin{equation}\label{eq:G-Ws-Ca}
|\Gamma_\sigma| + \norm{ W_\sigma }_{C^{2,\bar\alpha}(\T^n\times\T)} \leq C|E_{ij}| \leq C. 
\end{equation} 

Notice that Lemma \ref{lemma:cell-C0-C2a} ensures $w\in C(\cS^n;C^{2,\hat\alpha}(\T^n\times\T))$ for any $0<\hat\alpha<\bar\alpha$. This combined with uniform ellipticity \eqref{eq:ellip-F} of $F$ yields that we have $A_\sigma \ra A$ in $C^{\hat\alpha}(\T^n\times\T)$ as $\sigma\ra 0$ for any $0<\hat\alpha<\bar\alpha$, where 
\begin{equation}
A (y,s) = D_p F (D_y^2 w(P,y,s) + P,y,s). 
\end{equation}
On the other hand, the uniform estimate \eqref{eq:G-Ws-Ca} and the periodicity of $W_\sigma$ implies that any subsequence of $\{(\Gamma_\sigma,W_\sigma)\}_{\sigma\neq 0}$ contains a further subsequence that converges in $\R\times C^{2,\hat\alpha}(\T^n\times\T)$, for any $0<\hat\alpha<\bar\alpha$. However, the stability \cite{CIL} of viscosity solutions ensures that a (uniform) limit $(\Gamma,W)$ of $\{(\Gamma_\sigma,W_\sigma)\}_{\sigma\neq 0}$, if any, should satisfy 
\begin{equation*}
W_s = \tr( A(y,s) (D_y^2 W + E_{ij})) - \Gamma\quad\text{in }\T^n\times\T,
\end{equation*} 
in the viscosity sense. Since $A$ is periodic and uniformly elliptic (in the sense of \eqref{eq:As-ellip})  and $W$ is also periodic, the classical argument \cite{E2} ensures the uniqueness of $\Gamma$. Moreover, since $W_\sigma (0,0) = 0$ for all nonzero $\sigma$, the limit $W$ should also be unique. Therefore, $\Gamma_\sigma \ra \Gamma$ and $W_\sigma\ra W$ as $\sigma \ra 0$, where the latter holds in $C^{2,\hat\alpha}(\T^n\times\T)$ for any $0<\hat\alpha<\bar\alpha$. 

By the definition of $\Gamma_\sigma$ and $W_\sigma$, we conclude that $\bar{F}$ and $w$ are differentiable at $P$ in direction $E_{ij}$ with 
\begin{equation}
\Gamma = D_{p_{ij}} \bar{F}(P)\quad\text{and}\quad W(y,s) = D_{p_{ij}} w(P,y,s). 
\end{equation}
The rest of the proof then follows from Lemma \ref{lemma:cell-C0-C2a}, and hence we omit the details. 
\end{proof} 

The following proposition is obtained by induction on the order of derivatives of $\bar{F}$ and $w$ in the slow variables $(P,x,t)$.  

\begin{proposition}\label{proposition:cell-Ck-C2a} Assume that $F$ verifies \eqref{eq:ellip-F} - \eqref{eq:Ck-Ca-F}. Then $\bar{F} \in C^\infty(\cS^n\times\R^n\times[0,T])$ and $w\in C^\infty(\cS^n\times\R^n\times[0,T];C^{2,\bar\alpha}(\T^n\times\T))$ and
\begin{equation}\label{eq:cell-Ck-C2a}
\begin{split}
\sum_{k + l + 2r = m} \left[ \left| D_p^k D_x^l \p_t^r \bar{F} (P,x,t) \right| + \norm{ D_p^k D_x^l \p_t^r w(P,x,t,\cdot,\cdot) }_{C^{2,\bar\alpha}(\T^n\times\T)} \right] \leq C_m|P|^{(1-k)_+},
\end{split}
\end{equation}
for all $(P,x,t)\in\cS^n\times\R^n\times[0,T]$ and for each integer $m \geq 0$, where $0<\bar\alpha\leq\alpha$ depends only on $n$, $\lambda$, $\Lambda$ and $\alpha$, and $C_m>0$ depends only on $n$, $\lambda$, $\Lambda$, $\alpha$, $m$ and $K$.
\end{proposition}

\begin{proof}
One may notice that the higher regularity of $\bar{F}$ and $w$ in the slow variables $(P,x,t)$ can be obtained by inductively applying Lemma \ref{lemma:cell-C1-Ca} on the number of derivatives. Since the whole argument resembles that of the proof of Proposition \ref{proposition:decay-Ck-C2a}, we omit the details. 
\end{proof}


\section{Higher Order Convergence Rate}\label{section:higher}

This section is devoted to achieving the higher order convergence rates of the homogenization process of \eqref{eq:ue-pde}. We expect that away from the initial time zone, by which we indicate the strip $0\leq t\leq \e^2$, the solution, $u^\e$, of \eqref{eq:ue-pde} becomes less affected by the rapidly oscillatory behavior of the initial data, and that it behaves more as a solution to certain Cauchy problem with a non-oscillatory initial data. Thus, it is reasonable to split $u^\e$ into the non-oscillatory part and the oscillatory part near the initial time layer.

For this reason, we construct two types of the higher order correctors associated with the homogenization problem \eqref{eq:ue-pde}, namely the initial layer corrector and the interior corrector. The former type captures the oscillatory behavior of $u^\e$ near the initial time layer, while the latter describes its behavior in the interior. The construction of these correctors of higher orders will be based on the regularity theory in the slow variables established in Section \ref{section:slow}.

Recall from Section \ref{section:assump} the definition of the classes $S(m,\bar\alpha;d,k)$ and $E(m,\bar\alpha;d,k)$.


\subsection{Initial Layer Corrector}\label{subsection:ini}

In this subsection, we aim at proving the following proposition. 

\begin{proposition}\label{proposition:ini} 
Assume that $F : \cS^n\times\R^n\times[0,T]\times\T^n\times \T\ra \R$ and $g : \R^n\times\T^n\ra\R$ verify \eqref{eq:ellip-F} - \eqref{eq:Ck-C2a-g}. Then there exist sequences $\{v_k: \R^n\times[0,T]\times\T^n\times[0,\infty) \ra \R\}_{k=0}^\infty$ and $\{\bar{v}_k : \R^n\times[0,T]\ra\R\}_{k=0}^\infty$ such that $v_k \in S(2,\bar\alpha;k)$, $\bar{v}_k\in S(k)$ and $v_k - \bar{v}_k \in E(2,\bar\alpha;k)$ for each $k\geq 0$. Moreover, for any integer $m\geq 0$, set 
\begin{equation}\label{eq:vtme-gbme}
\tilde{v}_m^\e (x,t) = \sum_{k=0}^m \e^k \left( v_k \left(x,t,\frac{x}{\e},\frac{t}{\e^2}\right) - \bar{v}_k(x,t)\right),\quad \bar{g}_m^\e (x) = \sum_{k=0}^m \e^k \bar{v}_k (x,0).
\end{equation}
Then one has 
\begin{equation}\label{eq:vtme-pde}
\begin{dcases}
\p_t \tilde{v}_m^\e = \frac{1}{\e^2} F \left( \e^2 D^2 \tilde{v}_m^\e ,x,t,\frac{x}{\e},\frac{t}{\e^2}\right) + \psi_m^\e \left(x,t,\frac{x}{\e},\frac{t}{\e^2}\right) & \text{in }\R^n\times(0,T),\\
\tilde{v}_m^\e (x,0) + \bar{g}_m^\e (x) = g\left(x,\frac{x}{\e}\right) & \text{on }\R^n,
\end{dcases}
\end{equation}
where $\psi_m^\e : \R^n\times[0,T]\times\T^n\times[0,\infty)\ra \R$ satisfies 
\begin{equation}\label{eq:psime-decay}
\left| \psi_m^\e (x,t,y,s) \right| \leq C_m\e^{m-1} e^{-\beta_m s},
\end{equation}
for any $0<\e\leq\frac{1}{2}$. 
\end{proposition}

\begin{remark}\label{remark:ini-peri} 
We shall call $v_k(x,t,y,s)$ the $k$-th order initial layer corrector and the function $\bar{g}_k:\R^n\ra\R$, defined by 
\begin{equation}\label{eq:gbk}
\bar{g}_k(x) = \bar{v}_k(x,0),
\end{equation}
the $k$-th order effective initial data. 
\end{remark}

\begin{remark}\label{remark:ini-peri1}
If $F$ is a linear operator so that $F(P,x,t,y,s) = \tr(A(x,t,y,s)P)$ for some matrix-valued mapping $A$, then $\tilde{v}_m^\e$ satisfies
\begin{equation}\label{eq:vtme-pde-lin}
\begin{dcases}
\p_t \tilde{v}_m^\e = \tr \left(A \left( x,t,\frac{x}{\e},\frac{t}{\e^2}\right) D^2 \tilde{v}_m^\e \right) + \psi_m^\e \left(x,t,\frac{x}{\e},\frac{t}{\e^2}\right) & \text{in }\R^n\times(0,T),\\
\tilde{v}_m^\e (x,0) + \bar{g}_m^\e (x) = g\left(x,\frac{x}{\e}\right) & \text{on }\R^n.
\end{dcases}
\end{equation} 
\end{remark}

Let us begin with heuristic arguments by the formal expansion. The computation presented here uses the Taylor expansion of $F$ in its matrix variable $P$. We should mention that such an approach has already been shown in the authors' previous work \cite{KL1}. 

Differentiating $\tilde{v}_m^\e$ with respect to $t$, we obtain 
\begin{equation}\label{eq:vtme-t}
\e^2 \p_t \tilde{v}_m^\e (x,t) = \sum_{k=0}^m \e^k \p_s v_k \left(x,t,\frac{x}{\e},\frac{t}{\e^2}\right) + \sum_{k=0}^m \e^{k+2} \p_t \left(v_k \left(x,t,\frac{x}{\e},\frac{t}{\e^2}\right) - \bar{v}_k(x,t)\right).
\end{equation}
In order to proceed with the derivatives of $\tilde{v}_m^\e$ in variable $x$, let us  define $V_k : \R^n\times[0,T]\times\T^n\times[0,\infty) \ra \cS^n$ by 
\begin{equation}\label{eq:Vk}
V_k  = 
\begin{dcases}
D_y^2 v_0, & k =0,\\
D_y^2 v_1 + D_{xy} v_0, & k = 1,\\
D_y^2 v_k + D_{xy} v_{k-1}+ D_x^2 (v_{k-2} - \bar{v}_{k-2}), & k\geq 2,
\end{dcases}
\end{equation}
and corresponding define 
\begin{equation}\label{eq:Vtk}
\tilde{V}_k = 
\begin{dcases}
V_k, & 0\leq k\leq m,\\
D_{xy} v_m + D_x^2 (v_{m-1} - \bar{v}_{m-1}), & k=m+1,\\
D_x^2 (v_m - \bar{v}_m), & k = m+2.
\end{dcases}
\end{equation}
With $\tilde{V}_k$, one can write
\begin{equation}\label{eq:vtme-x}
\e^2 D^2 \tilde{v}_m^\e (x,t) = \sum_{k=0}^m \e^k  \tilde{V}_k \left(x,t,\frac{x}{\e},\frac{t}{\e^2}\right) = V_0 \left(x,t,\frac{x}{\e},\frac{t}{\e^2}\right) + \e V_m^\e \left(x,t,\frac{x}{\e},\frac{t}{\e^2}\right),
\end{equation} 
where in the second identity we wrote $V_m^\e$ for the sum of $\e^{k-1} \tilde{V}_k$ over $1\leq k\leq m+2$. 

For notational convenience, let us write
\begin{equation}\label{eq:Ak}
A_k (x,t,y,s) = D_p^k F (V_0, x,t,y,s),\quad k\geq 1, 
\end{equation}
for $(x,t,y,s)\in\R^n\times[0,T]\times\T^n\times[0,\infty)$, and especially 
\begin{equation*}
A = A_1,
\end{equation*}
which is a $\cS^n$-valued mapping, uniformly elliptic in the sense that $\tr(A(x,t,y,s) N)$ satisfies the ellipticity condition \eqref{eq:ellip-F}. What we shall do is the Taylor expansion of $F$ in the matrix variable with base point $V_0$ and perturbation $\e V_m^\e$. In the following computation, we shall omit the variables $(x,t,\e^{-1}x,\e^{-2}t)$ in the exposition, since they do not play any important role. 
\begin{equation}\label{eq:vtme-x-cont}
\begin{split}
F \left( \e^2 D^2 \tilde{v}_m^\e\right) &= F ( V_0 + \e V_m^\e ) \\
&= F ( V_0) + \e \tr (A V_m^\e ) + \sum_{k=2}^m \frac{\e^k}{k!} A_k  ( V_m^\e,\cdots,V_m^\e) + R_m^\e \\
&= F ( V_0) + \e \tr (A V_1)  + \sum_{k=2}^m \e^k \left( \tr(A V_k) + \sum_{l=2}^k \frac{1}{l!} \sum_{\substack{ i_1 + \cdots + i_l = k \\ i_1,\cdots,i_l \geq 1}} A_l( V_{i_1},\cdots, V_{i_l}) \right) \\
&\quad + E_m^\e, 
\end{split}
\end{equation}
where $R_m^\e$ is the remainder term from the Taylor expansion, i.e.,
\begin{equation}\label{eq:Rme}
R_m^\e = F (V_0 + \e V_m^\e ) - F ( V_0) - \e \tr (A V_m^\e ) - \sum_{k=2}^m \frac{\e^k}{k!} A_k  ( V_m^\e,\cdots,V_m^\e), 
\end{equation}
and $E_m^\e$ is the term that contains further errors,
\begin{equation}\label{eq:Eme}
\begin{split}
& E_m^\e = R_m^\e +  \sum_{k=2}^{m+2} \sum_{ \substack{m+1\leq i_1 + \cdots + i_k \leq k(m+2) \\ 1\leq i_1,\cdots, i_k\leq m+2 } }  \frac{\e^{i_1 + \cdots + i_k}}{k!} A_k (\tilde{V}_{i_1},\cdots,\tilde{V}_{i_k}),
\end{split}
\end{equation}
Hence, plugging $\tilde{v}_m^\e$ into \eqref{eq:ue-pde} and equating the power of $\e$, and noting \eqref{eq:vtme-t} and \eqref{eq:vtme-x-cont}, we obtain a sequence of equations that $v_k$ should solve. The next lemma gives a rigorous justification of the above heuristic arguments.

\begin{lemma}\label{lemma:ini} 
One can recursively construct sequences $\{v_k \in S(2,\bar\alpha;k)\}_{k=0}^\infty$ and $\{\bar{v}_k\in S(k)\}_{k=0}^\infty$, with $v_k - \bar{v}_k \in E(2,\bar\alpha;k)$, as follows.
\begin{enumerate}[(i)]
\item $v_0(x,t,\cdot,\cdot)$ is the solution of 
\begin{equation}\label{eq:v0-pde}
\begin{cases}
\p_s v_0 = F(D_y^2 v_0,x,t,y,s) & \text{in }\T^n\times(0,\infty),\\
v_0 (x,t,y,0) = g (x,y) & \text{on }\T^n,
\end{cases}
\end{equation} 
\item For each $2\leq k\leq m$, $v_k (x,t,\cdot,\cdot)$ is the solution of 
\begin{equation}\label{eq:vk-pde}
\begin{cases}
\p_s v_k = \tr(A(x,t,y,s) D_y^2 v_k) + \Phi_k(x,t,y,s) & \text{in }\T^n\times(0,\infty),\\
v_k(x,t,y,0) = 0 & \text{on }\T^n,
\end{cases}
\end{equation}
where 
\begin{equation}\label{eq:Phik-ini}
\Phi_k = 
\begin{dcases}
0, & k =1,\\
\begin{aligned}
&\tr(A(D_{xy} v_{k-1} + D_x^2 (v_{k-2} - \bar{v}_{k-2})) - \p_t (v_{k-2}  - \bar{v}_{k-2}) \\
&+ \sum_{l=2}^k \frac{1}{l!} \sum_{\substack{ i_1 + \cdots + i_l = k \\ i_1,\cdots,i_l \geq 1}} A_l( V_{i_1},\cdots, V_{i_l}),
\end{aligned}
& k\geq 2. 
\end{dcases}
\end{equation} 
\item For each $k\geq 0$, 
\begin{equation*}
\bar{v}_k(x,t) = \lim_{s\ra\infty} v_k(x,t,0,s). 
\end{equation*}
\end{enumerate}
\end{lemma}

\begin{remark}\label{remark:ini}
If $F$ were linear, i.e., $F(P,x,t,y,s) = \tr(A(x,t,y,s)P)$ for some matrix-valued mapping $A$, then the summation term in the definition \eqref{eq:Phik-ini} of $\Phi_k$, for $k\geq 2$, is trivial, since $A_l = D_p^l F = 0$ for any $l\geq 2$. Moreover, the interior operator for $v_0$ is the same as that of $v_k$ for any $k\geq 1$; i.e., $\partial_s v_0 = \tr(A(x,t,y,s) D_y^2 v_0)$ in $\T^n\times(0,\infty)$. 
\end{remark}

\begin{proof} It is clear from Proposition \ref{proposition:decay-Ck-C2a} that $v_0\in S(2,\bar\alpha;0)$ and $\bar{v}_0 \in S(0)$ with $v_0 - \bar{v}_0 \in E(2,\bar\alpha;0)$. Henceforth, we shall suppose $m\geq 1$, and assume further, as the induction hypothesis, that we have already found $v_k\in S(2,\bar\alpha;k)$ and $\bar{v}_k\in S(k)$ satisfying $v_k - \bar{v}_k \in E(2,\bar\alpha;k)$, for $0\leq k\leq m-1$. 

Recall the mappings $V_k$ and $A_k$ from \eqref{eq:Vk} and \eqref{eq:Ak}. Since $v_k - \bar{v}_k \in E(2,\bar\alpha;k)$, we have $V_k \in E(0,\bar\alpha;k)$ for each $0\leq k\leq m-1$. This along with the structure condition \eqref{eq:Ck-Ca-F} of $F$ that $A_k \in E(0,\bar\alpha;k)$ for each $k\geq 0$ as well. 

Now let $\Phi_m$ be as in \eqref{eq:Phik-ini}. One may notice that $\Phi_m$ only involves functions $v_k$ and $\bar{v}_k$, for $0\leq k\leq m-1$, which are assumed to be known already. Hence, combining the induction hypothesis that $v_k - \bar{v}_k \in E(2,\bar\alpha;k)$, and the observation that $V_k, A_k\in E(0,\bar\alpha;k)$, deduce that $\Phi_m \in E(0,\bar\alpha;m)$. Thus, one can apply Proposition \ref{proposition:decay-Ck-C2a} again to the viscosity solution $v_m(x,t,\cdot,\cdot)$ of \eqref{eq:vk-pde}, and verify that $v_m\in S(2,\bar\alpha;m)$, $\bar{v}_m \in S(2,\bar\alpha;m)$ and $v_m - \bar{v}_m \in E(2,\bar\alpha;m)$. The proof is then completed by the induction principle. 
\end{proof}

\begin{remark}\label{remark:ini-proof} Let us remark that the proof above does not involve the periodicity of $F$ in the fast temporal variable $s$. This is why Proposition \ref{proposition:ini} holds even if we only assume the spatial periodicity of $F$ (that is periodicity in $y$), as mentioned in Remark \ref{remark:ini-peri}.
\end{remark}

We are now ready to prove Proposition \ref{proposition:ini}

\begin{proof}[Proof of Proposition \ref{proposition:ini}]
Let $\{v_k\}_{k=0}^\infty$ and $\{\bar{v}_k\}_{k=0}^\infty$ be the sequence taken from Lemma \ref{lemma:ini}, and let $\tilde{v}_m^\e$ and $\bar{g}_m^\e$ be as in \eqref{eq:vtme-gbme}.  Then it follows from \eqref{eq:vtme-t}, \eqref{eq:vtme-x-cont}, \eqref{eq:v0-pde} and \eqref{eq:vk-pde} that the functions $\tilde{v}_m^\e$ and $\bar{g}_m^\e$ defined by \eqref{eq:vtme-gbme} satisfy \eqref{eq:vtme-pde} with
\begin{equation*}
\psi_m^\e (x,t,y,s)  = \sum_{k= m-1}^m \e^k \p_t (v_k (x,t,y,s) - \bar{v} (x,t)) - \e^{-2} E_m^\e (x,t,y,s),
\end{equation*}
where $E_m^\e$ is given by \eqref{eq:Eme}. The rest of the proof is devoted to the proof of \eqref{eq:psime-decay}. 

From the fact that $v_k - \bar{v}_k \in E(2,\bar\alpha;k)$ for any $k\geq 0$, we know that
\begin{equation}\label{eq:vtm-decay}
\sum_{k= m-1}^m \e^k \left| \p_t (v_k (x,t,y,s) - \bar{v}_k(x,t)) \right| \leq C_m \e^{m-1} e^{-\beta_m s},
\end{equation}
for any $0<\e\leq\frac{1}{2}$. On the other hand, from the observation that $V_k, A_k \in E(0,\bar\alpha;k)$, the remainder term $R_m^\e$ in \eqref{eq:Rme} can be estimated as 
\begin{equation}\label{eq:Rm-decay}
\left| R_m^\e (x,t,y,s)\right| \leq \frac{\e^{m+1}}{(m+1)!} \left|B_{m+1}(V_m^\e,\cdots,V_m^\e)\right|(x,t,y,s)\leq C_m \e^{m+1} e^{-\beta_m s},
\end{equation}
for any $0<\e\leq\frac{1}{2}$. Noting that the summation indices $i_1,\cdots,i_k$ in the definition of $E_m^\e$ are subject to the restriction $i_1+ \cdots + i_k \geq m+1$, we deduce from \eqref{eq:Rm-decay} that
\begin{equation}\label{eq:Eme-decay}
\left| E_m^\e (x,t,y,s) \right| \leq C_m \e^{m+1} e^{-\beta_m s},
\end{equation} 
for all $0<\e\leq \frac{1}{2}$. Thus, \eqref{eq:psime-decay} follows from \eqref{eq:vtm-decay} and \eqref{eq:Eme-decay}. This completes the proof. 
\end{proof}


\subsection{Interior Corrector}\label{subsection:int}

In this subsection, we shall construct the higher order interior correctors. Here we shall consider a more general class of homogenization problems compared to \eqref{eq:ue-pde}. This will be essential in order to achieve the higher order convergence rate away from the initial time layer, and we shall discuss more in this direction in Section \ref{subsection:bootstrap}. 

\begin{proposition}\label{proposition:int} 
Assume that $F : \cS^n\times\R^n\times[0,T]\times\T^n\times\T\ra \R$ verifies \eqref{eq:ellip-F} - \eqref{eq:Ck-Ca-F}. Let $\{V_k : \R^n\times[0,T]\times\T^n\times[0,\infty) \ra \cS^n \}_{k=0}^\infty$ and $\{\bar{g}_k : \R^n\ra\R\}_{k=0}^\infty$ be such that $V_k \in E(0,\bar\alpha;k)$ and $\bar{g}_k\in S(k)$.

Then there exist  $\{w_k : \R^n\times[0,T]\times\T^n\times[0,\infty) \ra \R \}_{k=0}^\infty$, $\{w_k^\# : \R^n\times[0,T]\times\T^n\times\T \ra \R\}_{k=0}^\infty$ and $\{\bar{u}_k : \R^n\times[0,T]\ra\R\}_{k=0}^\infty$ satisfying $w_k, w_k^\# \in S(2,\bar\alpha;k)$, $w_k - w_k^\# \in E(2,\bar\alpha;k)$ with $w_0 = w_1 = w_0^\# = w_1^\# = 0$, as well as $\bar{u}_k\in S(k)$ with $\bar{u}_k(\cdot,0) = \bar{g}_k$ for any $k\geq 0$, such that the following is true. Define 
\begin{align*}
\tilde{w}_m^\e (x,t) &= \sum_{k=0}^m \e^k \left(w_k \left(x,t,\frac{x}{\e},\frac{t}{\e^2}\right) + \bar{u}_k(x,t) \right),\\
\tilde{w}_m^{\#,\e} (x,t) &= \sum_{k=0}^m \e^k \left(w_k^\# \left(x,t,\frac{x}{\e},\frac{t}{\e^2}\right) + \bar{u}_k(x,t) \right).
\end{align*}
Then one has 
\begin{equation}\label{eq:wtme-pde}
\begin{split}
\p_t \tilde{w}_m^\e &= \frac{1}{\e^2} F \left( \e^2 D^2 \tilde{w}_m^\e + \sum_{k=0}^m \e^k V_k,x,t,\frac{x}{\e},\frac{t}{\e^2}\right) \\
& \quad - \frac{1}{\e^2} F \left( \sum_{k=0}^m \e^k V_k,x,t,\frac{x}{\e},\frac{t}{\e^2}\right) + \psi_m^\e \left(x,t,\frac{x}{\e},\frac{t}{\e^2}\right) \quad \text{in }\R^n\times(0,T),
\end{split}
\end{equation}
where $V_k$ is evaluated at $(x,t,\e^{-1}x,\e^{-2}t)$, and 
\begin{equation}\label{eq:wtmeh-pde}
\partial_t \tilde{w}_m^{\#,\e} = \frac{1}{\e^2} F \left( \e^2 D^2 \tilde{w}_m^{\#,\e},x,t,\frac{x}{\e},\frac{t}{\e^2}\right) + \psi_m^{\#,\e} \left(x,t,\frac{x}{\e},\frac{t}{\e^2}\right) \quad\text{in }\R^n\times(0,T),
\end{equation}
with some $\psi_m^\e:\R^n\times[0,T]\times\T^n\times[0,\infty)\ra \R$, $\psi_m^{\#,\e}:\R^n\times[0,T]\times\T^n\times\T\ra \R$ satisfying 
\begin{equation}\label{eq:psime-psimhe-decay-Linf}
\left| \psi_m^{\#,\e} (x,t,y,s) \right|  + e^{\beta_m s} \left| (\psi_m^\e - \psi_m^{\#,\e}) (x,t,y,s)\right| \leq C_m\e^{m-1},
\end{equation} 
for any $0<\e\leq\frac{1}{2}$. 
\end{proposition}

\begin{remark}\label{remark:int}
The function $w_k^\#$ will be a time-periodic version of $w_k$, i.e., the former is also periodic in the fast time variable $s$ as well as $y$. In what follows, we shall call $w_k$ the $k$-th order interior corrector, and $\bar{u}_k$ the $k$-th order effective limit profile. 
\end{remark}

\begin{remark}\label{remark:int2}
The initial condition of $w_k$ is fixed by the condition $w_k - w_k^\#\in E(2,\bar\alpha;k)$. The precise condition will be specified in Lemma \ref{lemma:int} below. Moreover, $w_k$ and $w_k^\#$ depend on $\bar{g}_{k-2}$ inductively for every $k\geq 2$. This will also be made clear in the statement of Lemma \ref{lemma:int}. 
\end{remark} 

\begin{remark}\label{remark:int3}
Let us also address that the reason we consider not only $w_k$ but also $w_k^\#$ is due to the nonlinearity of the operator $F$. If $F$ were linear, i.e., $F(M,x,t,y,s) = \tr(A(x,t,y,s)M)$ for some matrix $A$, then we would have $w_k = w_k^\#$, implying that the higher order interior correctors are periodic in both space and time variables. Moreover, the construction of $w_k$ in that case will only depend on the given initial data $\bar{g}_k$, but not $V_k$, which affects the interior oscillation in the nonlinear case. These facts will be explained in more details later in Remark \ref{remark:chi2} and Remark \ref{remark:int5}. Furthermore, in the linear case, $\tilde{w}_m^\e = \tilde{w}_m^{\#,\e}$, $\psi_m^\e = \psi_m^{\#,\e}$ and $\tilde{w}_m^\e$ solves
\begin{equation}\label{eq:wtme-pde-lin}
\p_t \tilde{w}_m^\e = \tr \left(A\left(x,t,\frac{x}{\e},\frac{t}{\e^2}\right) \e^2 D^2 \tilde{w}_m^\e\right) + \psi_m^\e \left(x,t,\frac{x}{\e},\frac{t}{\e^2}\right) \quad \text{in }\R^n\times(0,T).
\end{equation}
\end{remark} 

Let us begin with the construction of matrix-corrector corresponding to the coefficient $A(x,t,\cdot,\cdot)$ defined as in \eqref{eq:Ak}, for each $(x,t)\in\R^n\times[0,T]$. Recall that $A(x,t,y,s) = D_P F(V_0,x,t,y,s)$, with $V_0=V_0(x,t,y,s)$ as in \eqref{eq:Vk}, and also that $V_0(x,t,y,s)$ decays exponentially fast as the time variable $s$ tends to $\infty$, while it oscillates periodically in the space variable $y$. For this reason, $A(x,t,y,s)$ becomes exponentially close to a space-time periodic matrix-valued mapping, $D_p F(0,x,t,y,s)$, as $s\ra\infty$, and the corresponding matrix-corrector should also capture such an oscillating pattern. 

Define $A^\#:\R^n\times[0,T]\times\T^n\times\T\ra\cS^n$ by 
\begin{equation}\label{eq:Ah}
A^\#(x,t,y,s) = D_p F (0,x,t,y,s). 
\end{equation}
As noted in the preceding paragraph, $A^\#$ is the space-time periodic matrix-valued mapping corresponding to $A$. Due to the periodicity in both space and time, to $A^\#(x,t,\cdot,\cdot)$ there corresponds a unique effective coefficient $\bar{A}(x,t)\in\cS^n$, for each $(x,t)\in\R^n\times[0,T]$. Moreover, due to \cite{E2}, there exists a unique $\cS^n$-valued mapping $\chi^\# = (\chi_{ij}^\#)$ on $\R^n\times[0,T]\times\T^n\times\T$ such that for each $(x,t)\in\R^n\times[0,T]$, $\chi_{ij}^\#(x,t,\cdot,\cdot)$ is the unique periodic solution to the following cell problem,
\begin{equation}\label{eq:chih-pde}
\begin{cases}
\p_s \chi_{ij}^\# = \tr (A^\# (x,t,y,s) (D_y^2 \chi_{ij}^\# + E_{ij})) - \tr (\bar{A} (x,t) E_{ij})\quad\text{in }\T^n\times\T, \\
\chi_{ij}^\# (x,t,0,0) = 0.
\end{cases}
\end{equation}
In particular, $\bar{A}$ is uniformly elliptic with the same ellipticity bounds as those of $A^\#$. Moreover by Proposition \ref{proposition:cell-Ck-C2a}, we know that $\bar{A} \in S(0)$, $\chi^\# \in S(2,\alpha;0)$, where $\alpha$ is the H\"{o}lder exponent in the regularity assumption \eqref{eq:Ck-Ca-F} of $F$. 

The following lemma ensures the existence of a matrix corrector mapping that exactly captures the oscillatory behavior of the coefficient $A$. 

\begin{lemma}\label{lemma:chi} 
There are mappings $\chi = (\chi_{ij}):\R^n\times[0,T]\times\T^n\times[0,\infty) \ra \cS^n$ and $\bar\vp = (\bar\vp_{ij}) : \R^n\times[0,T]\ra \cS^n$ such that $\chi_{ij}(x,t,\cdot,\cdot)$ is the solution to
\begin{equation}\label{eq:chi-pde}
\begin{cases}
\p_s \chi_{ij} = \tr (A(x,t,y,s) (D_y^2 \chi_{ij} + E_{ij})) - \tr (\bar{A}(x,t) E_{ij}) & \text{in }\T^n\times(0,\infty),\\
\chi_{ij}(x,t,y,0)=  \chi_{ij}^\# (x,t,y,0) - \bar\vp_{ij}(x,t) & \text{on }\T^n,
\end{cases}
\end{equation}
for each $(x,t)\in\R^n\times[0,T]$, and that $\chi - \chi^\# \in E(2,\bar\alpha;0)$. Moreover, $\chi \in S(2,\bar\alpha;0)$ and $\bar\vp\in S(0)$. 
\end{lemma}

\begin{remark}\label{remark:chi} 
One can replace $\chi_{ij}^\#$ in the initial condition of \eqref{eq:chi-pde} with an arbitrary (smooth) spacially periodic data, but still obtain the relation $\chi_{ij} - \chi_{ij}^\# \in E(2,\bar\alpha;0)$ by modifying $\bar\vp_{ij}$ accordingly in the same initial condition. Here we choose $\chi_{ij}^\#$ in the initial condition in order to have $\chi_{ij} = \chi_{ij}^\#$ for linear problems, which is briefly explained in the next remark.
\end{remark}

\begin{remark}\label{remark:chi2}
Suppose that $F$ is a linear operator, that is, $F(P,x,t,y,s) = \tr(A(x,t,y,s)P)$. This implies in \eqref{eq:Ak} and \eqref{eq:Ah} that $A (x,t,y,s) = D_p F(V_0,x,t,y,s) = D_p F(0,x,t,y,s) = A^\# (x,t,y,s)$, so the interior equations for $\chi_{ij}$ and $\chi_{ij}^\#$ are the same. Now the requirement $\chi - \chi^\# \in E(2,\bar\alpha;0)$ forces $\chi = \chi^\#$. 
\end{remark}

\begin{proof}[Proof of Lemma \ref{lemma:chi}]
Fix $(x,t)\in\R^n\times[0,T]$, $1\leq i,j\leq n$ and consider the following spatially periodic Cauchy problem, 
\begin{equation*}
\begin{cases}
\p_s \vp_{ij} = \tr (A(x,t,y,s) D_y^2 \vp_{ij} ) + b_{ij} (x,t,y,s) & \text{in }\T^n\times(0,\infty),\\
\vp_{ij}(x,t,y,0)=  0 & \text{on }\T^n,
\end{cases}
\end{equation*}
with 
\begin{equation*}
b_{ij} (x,t,y,s) = \tr ( (A (x,t,y,s) - A^\#(x,t,y,s) )(D_y^2 \chi_{ij}^\#(x,t,y,s) + E_{ij}) ).
\end{equation*}
Since $V_k \in E(0,\bar\alpha;k)$ and $\chi^\# \in S(k)$, we know that $b_{ij}$ decays exponentially fast as $s\ra\infty$. Thus, Lemma \ref{lemma:decay} implies that the function, 
\begin{equation*}
\bar\vp_{ij}(x,t) = \lim_{s\ra\infty} \vp_{ij} (x,t,0,s),
\end{equation*}
is well defined. Now it follows from the regularity assumption \eqref{eq:Ck-Ca-F} on $F$ together with Proposition \ref{proposition:decay-Ck-C2a} that this lemma is satisfied by 
\begin{equation*}
\chi_{ij} (x,t,y,s) = \chi_{ij}^\# (x,t,y,s) + (\vp_{ij} (x,t,y,s) - \bar\vp_{ij} (x,t)).
\end{equation*}
We omit the details.
\end{proof}

In what follows, let us write 
\begin{equation}\label{eq:Wk}
W_k = 
\begin{dcases}
0, & k=0,1,\\
D_y^2 w_k + D_{xy} w_{k-1} + D_x^2 (w_{k-2} + \bar{u}_{k-2}), & k\geq 2.
\end{dcases}
\end{equation}
Note that we set $W_0 = W_1 = 0$, which is coherent the assertion in Proposition \ref{proposition:int} that $w_0 = w_1 = 0$. Next set $W_k^\#$ by the time-periodic version of $W_k$, that is, 
\begin{equation}\label{eq:Wkh}
W_k^\# = 
\begin{dcases}
0, & k=0,1,\\
D_y^2 w_k^\# + D_{xy} w_{k-1}^\# + D_x^2 (w_{k-2}^\# + \bar{u}_{k-2}), & k\geq 2.
\end{dcases}
\end{equation}
Also let $A_k$ be as in \eqref{eq:Ak}, and set $A_k^\# \in S(0,\alpha;k)$ to its time-periodic version, 
\begin{equation}\label{eq:Akh}
A_k^\# (x,t,y,s) = D_p^k F(0,x,t,y,s),\quad k\geq 1.
\end{equation}
It follows from $V_0 \in S(0,\bar\alpha;0)$ that $A_k - A_k^\# \in E(0,\bar\alpha;k)$ for any $k\geq 0$. 

We are now ready to construct the higher order interior correctors as follows. 

\begin{lemma}\label{lemma:int} 
One can recursively construct $\{w_k:\R^n\times[0,T]\times\T^n\times[0,\infty)\ra\R \}_{k=0}^\infty$, $\{w_k^\#:\R^n\times[0,T]\times\T^n\times\T\ra \R\}_{k=0}^\infty$, $\{\bar{h}_k, \bar{u}_k : \R^n\times[0,T]\ra\R \}_{k=0}^\infty$ such that $w_k,w_k^\#\in S(2,\bar\alpha;k)$, $w_k - w_k^\# \in E(2,\bar\alpha;k)$, $\bar{h}_k, \bar{u}_k \in S(k)$ and the following hold. 
\begin{enumerate}[(i)]
\item $w_k^\#(x,t,\cdot,\cdot)$ is the periodic solution to 
\begin{equation}\label{eq:wkh-pde}
\begin{dcases}
\p_s w_k^\# = \tr( A^\# (x,t,y,s) D_y^2 w_k^\#) + \Phi_k^\# (x,t,y,s)\quad \text{in }\T^n\times\T,\\
w_k^\# (x,t,0,0) = 0,
\end{dcases}
\end{equation}
for each $(x,t)\in\R^n\times[0,T]$, where 
\begin{equation*}
\Phi_k^\# =
\begin{dcases}
0, & k=0,1,\\
\begin{aligned}
&\tr(A^\#(D_{xy} w_{k-1}^\# + D_x^2 (w_{k-2}^\# + \bar{u}_{k-2}))) - \p_t (w_{k-2}^\# + \bar{u}_{k-2}) \\
& + \sum_{l=2}^k \frac{1}{l!} \sum_{\substack{i_1 + \cdots + i_l =k \\ i_1,\cdots,i_l\geq 1 }} A_l^\#(W_{i_1}^\#,\cdots, W_{i_l}^\#).
\end{aligned}
& k\geq 2.
\end{dcases}
\end{equation*}
\item $w_k(x,t,\cdot,\cdot)$ is the solution to 
\begin{equation}\label{eq:wk-pde}
\begin{cases}
\p_s w_k  = \tr ( A(x,t,y,s) D_y^2 w_k) + \Phi_k (x,t,y,s) & \text{in }\T^n\times(0,\infty), \\
w_k (x,t,y,0) = w_k^\#(x,t,y,0) - \bar{h}_k(x,t) & \text{on }\T^n,
\end{cases}
\end{equation}
for each $(x,t)\in\R^n\times[0,T]$, where 
\begin{equation*}
\Phi_k =
\begin{dcases}
0, & k=0,1,\\
\begin{aligned}
&\tr(A(D_{xy} w_{k-1} + D_x^2 (w_{k-2} + \bar{u}_{k-2}))) - \p_t (w_{k-2} + \bar{u}_{k-2}) \\
& + \sum_{l=2}^k \frac{1}{l!} \sum_{\substack{i_1 + \cdots + i_l =k \\ i_1,\cdots,i_l\geq 1 }} (A_l(V_{i_1} + W_{i_1},\cdots, V_{i_l} + W_{i_l}) -A_l(V_{i_1},\cdots, V_{i_l})),
\end{aligned}
& k\geq 2.
\end{dcases}
\end{equation*}
\item $\bar{u}_k$ is the unique solution of
\begin{equation}\label{eq:ubk-pde}
\begin{cases}
\p_t \bar{u}_k = \tr (\bar{A}(x,t) D_x^2 \bar{u}_k) + \bar\Phi_k(x,t) & \text{in }\R^n\times(0,T),\\
\bar{u}_k (x,0) = \bar{g}_k(x) & \text{on }\R^n,
\end{cases}
\end{equation} 
where $\bar\Phi_k(x,t)$ is the unique number for which there exists a solution to 
\begin{equation}\label{eq:phikh-pde}
\begin{cases}
\p_s \phi_k^\#= \tr (A^\#(x,t,y,s) D_y^2 \phi_k^\#) + \Phi_k^\#(x,t,y,s) - \bar\Phi_k(x,t) \quad\text{in }\T^n\times\T, \\
\phi_k^\# (x,t,0,0) =0,
\end{cases}
\end{equation} 
for each $(x,t)\in\R^n\times[0,T]$. 
\end{enumerate}
\end{lemma}

\begin{remark}\label{remark:int4}
As mentioned in Remark \ref{remark:int2}, $w_k$ and $w_k^\#$ depend on $\bar{g}_{k-2}$ (and more accurately, all $\bar{g}_l$ with $0\leq l\leq k-2$) for all $k\geq 2$. This is because $\Phi_k$ and $\Phi_k^\#$ involve $\bar{u}_{k-2}$ (and again $\bar{u}_l$ for any $0\leq l\leq k-2$), and $\bar{u}_{k-2} = \bar{g}_{k-2}$ on the initial layer $\R^n\times\{0\}$, as is shown in \eqref{eq:ubk-pde}. 
\end{remark}

\begin{remark}\label{remark:int5}
If $F$ is a linear operator, then $w_k = w_k^\#$, since the interior equations for these functions are the same. This can be shown as follows. First, recall from Remark \ref{remark:chi2} that $A = A^\#$, due to the linearity of $F$. Moreover, we have $A_l = A_l^\# = 0$ for any $l\geq 2$, since both $A_l$ and $A_l^\#$ involve the $l$-th order derivative of $F$ in the matrix variable only. Now assuming that $w_l = w_l^\#$ for any $0\leq l\leq k-1$, we have $\Phi_k = \Phi_k^\#$ as well, since both $\Phi_k$ and $\Phi_k^\#$ involve $w_l$ and respectively $w_l^\#$ for $l\leq k-1$ only. This proves that the interior equations for $w_k$ and $w_k^\#$ are the same if $F$ is a linear operator. Now the requirement that $w_k - w_k^\# \in E(2,\bar\alpha;k)$ forces $w_k = w_k^\#$. 

In addition, the interior equation for $w_k$ in the linear case does not involve the data $V_k$ (or any of $\{V_l\}_{l=0}^\infty$), which shows that the construction of $w_k$ depends only on the given initial data $\bar{g}_k$. 
\end{remark}

\begin{proof} 
Since $\Phi_k^\#= 0$ for $k=0,1$, one should have $w_k^\# = 0$ for $k=0,1$ as well, since $w_k^\#$ is the unique periodic solution to \eqref{eq:wkh-pde}. This also implies that $w_k = 0$ for $k=0,1$. Hence, we only need to construct $w_k^\#$ and $w_k$ for $k\geq 2$.

Let us remark that the construction of $w_k^\#$ and $\bar{u}_{k-2}$, for $k\geq 2$, is independent of $w_k$. Moreover, the construction is very similar with the elliptic case, which can be found in the previous work \cite[Lemma 3.3.2]{KL1} by the authors. Especially, $w_k^\#$ is given by 
\begin{equation*}
w_k^\# (x,t,y,s) = \phi_k^\# (x,t,y,s) + \tr \left( \chi^\#(x,t,y,s) D_x^2 \bar{u}_{k-2}(x,t) \right),
\end{equation*} 
with $\chi^\#$ and $\phi_k^\#$ given as the unique periodic solutions to \eqref{eq:chih-pde} and respectively \eqref{eq:phikh-pde}; here one can also deduce that $\phi_k^\#\in S(2,\bar\alpha;k)$. We shall leave this part to the reader, and proceed directly with the construction of $w_k$ only. 

Fix any $m\geq 2$, and suppose that we have already found $w_k^\#$, $\phi_k^\#$ and $\bar{u}_{k-2}$, for $k\leq m$, and $w_k$, for $k\leq m-1$, that satisfy the assertions of this lemma. Note that $\Phi_m$ only involves these functions. Since $w_k\in S(2,\bar\alpha;k)$ and $w_k^\#\in S(2,\bar\alpha;k)$ together satisfy $w_k - w_k^\# \in E(2,\bar\alpha;k)$ as the induction hypotheses for $0\leq k\leq m-1$, one can derive along with the assumption $V_k\in S(0,\bar\alpha;k)$ that $\Phi_m \in S(0,\bar\alpha;m)$ and $\Phi_m - \Phi_m^\# \in S(0,\bar\alpha;m)$. Hence, one can argue analogously as with the proof of Lemma \ref{lemma:chi} and obtain functions $\bar\psi_m \in S(m)$ and $\phi_m \in S(2,\bar\alpha;m)$ such that for each $(x,t)\in\R^n\times[0,T]$, $\phi_m(x,t,\cdot,\cdot)$ is the solution of 
\begin{equation}\label{eq:phim-pde}
\begin{cases}
\p_s \phi_m = \tr (A (x,t,y,s) D_y^2 \phi_m) + \Phi_m (x,t,y,s) - \bar\Phi_m (x,t) &\text{in }\T^n\times(0,\infty),\\
\phi_m(x,t,y,0) = \phi_m^\#(x,t,y,0) - \bar\psi_m(x,t) & \text{on }\T^n, 
\end{cases}
\end{equation}
and that $\phi_m - \phi_m^\# \in E(2,\bar\alpha;m)$; this inclusion follows from $\Phi_m - \Phi_m^\# \in E(0,\bar\alpha;m)$ and $A - A^\# \in E(0,\bar\alpha;0)$. 

Finally, we define $w_m:\R^n\times[0,T]\times\T^n\times[0,\infty)\ra\R$ by 
\begin{equation*}
w_m (x,t,y,s) = \phi_m (x,t,y,s) + \tr (\chi (x,t,y,s) D_x^2 \bar{u}_{m-2} (x,t)),
\end{equation*}
and 
\begin{equation*}
\bar{h}_m(x,t) = \bar\psi_m (x,t) + \tr ( \bar\vp (x,t) D_x^2 \bar{u}_{m-2}(x,t)),
\end{equation*}
where $\bar\vp$ is the $\cS^n$-valued mapping chosen from Lemma \ref{lemma:chi}. Then it follows from \eqref{eq:phim-pde} and \eqref{eq:chi-pde} that
\begin{equation*}
\begin{split}
\p_s w_m &= \p_s \phi_m + \tr ( (\p_s \chi)D_x^2 \bar{u}_{m-2})\\
& = \tr(AD_y^2 \phi_m) + \Phi_m - \bar\Phi_m + \tr((AD_y^2 \chi - \bar{A})D_x^2 \bar{u}_{m-2}) \\
& = \tr(AD_y^2 w_m) + \Phi_m,
\end{split}
\end{equation*}
in $(y,s)$, with each $(x,t)$ fixed, which verifies the interior equation in \eqref{eq:wk-pde}. To justify the initial condition, observe from the initial condition of \eqref{eq:chi-pde} and \eqref{eq:phim-pde} that 
\begin{equation*}
\begin{split}
w_m (x,t,y,0) &= \phi_m^\#(x,t,y,0) - \bar\psi_m(x,t) + \tr ( (\chi^\# (x,t,y,0) - \bar\vp(x,t)) D_x^2 \bar{u}_{m-2}(x,t)) \\
&= w_m^\#(x,t,y,0) - \bar{h}_m(x,t). 
\end{split}
\end{equation*}
On the other hand, $w_m \in S(2,\bar\alpha;m)$, $\bar{h}_m \in S(2,\bar\alpha;m)$ and $w_m - w_m^\#\in E(2,\bar\alpha;m)$, since we have $\phi_m - \phi_m^\# \in E(2,\bar\alpha;m)$, $\chi- \chi^\# \in E(2,\bar\alpha;0)$, $\bar\psi_m \in S(m)$, $\bar\vp \in S(0)$ and $\bar{u}_{m-2} \in S(m-2)$. We omit the details. 
\end{proof}

Equipped with Lemma \ref{lemma:int}, we are ready to prove Proposition \ref{proposition:int}. 

\begin{proof}[Proof of Proposition \ref{proposition:int}] 
Here we shall only address the notable difference in the computation involving the Taylor expansion, when proving Proposition \ref{proposition:int} (iv), and leave the rest of the argument to the reader, since the main argument follows closely to the proof of Proposition \ref{proposition:ini}. Let us define
\begin{equation*}
X_k = V_k + W_k, 
\end{equation*}
with $W_k$ given as in \eqref{eq:Wk}. Since $W_0 = W_1 = 0$, we have $X_0 = V_0$ and $X_1 = V_1$. Thus, one can proceed as in the computation in \eqref{eq:vtme-x-cont} and deduce that 
\begin{equation*}
\begin{split}
& F ( V_0 + \e X_m^\e ) - F( V_0 + \e V_m^\e) + \psi_m^\e\\
& = \sum_{k=2}^m \e^k \left( \tr(A(X_k - V_k)) + \sum_{l=2}^k \frac{1}{l!} \sum_{ \substack{ i_1+\cdots + i_l = k  \\ i_1,\cdots,i_l \geq 1}} ( A_l (X_{i_1},\cdots,X_{i_l}) - A_l (V_{i_1},\cdots,V_{i_l}) ) \right)  \\
&= \sum_{k=2}^m \e^k \left( \tr(AW_k) + \sum_{l=2}^k \frac{1}{l!} \sum_{ \substack{ i_1+\cdots + i_l = k  \\ i_1,\cdots,i_l \geq 1}} ( A_l (V_{i_1} + W_{i_1},\cdots,V_{i_l} + W_{i_l}) - A_l (V_{i_1},\cdots,V_{i_l}) ) \right),
\end{split}
\end{equation*}
where $\psi_m^\e$ is the error term of the form \eqref{eq:Eme}, $V_m^\e = \sum_{k=1}^m \e^{k-1} V_k$ and $X_m^\e= \sum_{k=1}^m \e^{k-1} X_k + \sum_{k=m+1}^{m+2} \e^{k-1} \tilde{W}_k$, with 
\begin{equation}\label{eq:Wtk}
\tilde{W}_k = 
\begin{dcases}
W_k, & 0\leq k\leq m,\\
D_{xy} w_m + D_x^2 (w_{m-1} + \bar{u}_{m-1}), & k = m+1,\\
D_x^2 (w_m + \bar{u}_m), & k = m+2. 
\end{dcases}
\end{equation}
Hence, it follows from the recursive equations \eqref{eq:wk-pde} of $w_k$ that 
\begin{equation*}
F (V_0 + \e X_m^\e ) - F ( V_0 + \e V_m^\e) - E_m^\e = \sum_{k=2}^m \e^k (\p_s w_k + \p_t (w_{k-2} + \bar{u}_{k-2})). 
\end{equation*}
This shows that $\tilde{w}_m^\e$ solves \eqref{eq:wtme-pde} with the remainder term $\psi_m^\e$. We skip the rest of the proof. 
\end{proof}

\subsection{Iteration Scheme and Nonlinear Coupling Effect}\label{subsection:bootstrap}

The main goal of this subsection is to establish our main result, Theorem \ref{theorem:higher}. Throughout this subsection, let us assume that $F:\cS^n\times\R^n\times[0,T]\times\T^n\times\T\ra \R$ and $g:\R^n\times\T^n\ra\R$ verify \eqref{eq:ellip-F}  - \eqref{eq:Ck-C2a-g}. See Figure \ref{figure:iter} for a diagram on the iteration scheme for higher order correction.

\usetikzlibrary{calc,trees,positioning,arrows,chains,shapes.geometric,%
    decorations.pathreplacing,decorations.pathmorphing,shapes,%
    matrix,shapes.symbols}

\tikzset{
>=stealth',
  punktchain/.style={
    rectangle, 
    rounded corners, 
    draw=black, very thick,
    text width=12em, 
    minimum height=3em, 
    text centered, 
    on chain},
  punktchain2/.style={
    rectangle, 
    rounded corners, 
    draw=black, very thick,
    text width=8em, 
    minimum height=3em, 
    text centered, 
    on chain},
  line/.style={draw, thick, <-},
  element/.style={
    tape,
    top color=white,
    bottom color=blue!50!black!60!,
    minimum width=8em,
    draw=blue!40!black!90, very thick,
    text width=10em, 
    minimum height=3.5em, 
    text centered, 
    on chain},
  every join/.style={->, thick,shorten >=1pt},
  decoration={brace},
  tuborg/.style={decorate},
  tubnode/.style={midway, right=2pt},
}

\begin{figure}
\begin{tikzpicture}
  [node distance=.8cm,
  start chain=going below,]
\node[punktchain, join] (intro) {Homogenization of  
     \begin{equation*}
     (\partial_t - F_{d,m}^\e)  = 0
     \end{equation*}
     with rapidly oscillating \\initial data
     \begin{equation*}
     g_{d,m}^\e \left(\cdot,\frac{\cdot}{\e}\right)
     \end{equation*}};
     
\node[punktchain, join] (ini) {Initial layer correction
     \begin{equation*}
     \begin{split}
     u_{d,m}^\e &= \tilde{v}_{d,m}^\e + \tilde{u}_{d,m}^\e \\
     &\quad + O(\e^{m-1})
     \end{split}
     \end{equation*}};
     
      \begin{scope}[start branch=venstre,
        every join/.style={->, thick, shorten <=1pt}, ]
        \node[punktchain2, on chain=going left]
            (scale) {Normalization
            \begin{equation*}
            \e^2 u_{d+1,m}^\e \mapsto u_{d+1,m}^\e 
            \end{equation*}};
     \end{scope}

      \begin{scope}[start branch=venstre,
        every join/.style={->, thick, shorten <=1pt}, ]
        \node[punktchain2, on chain=going right]
            (linear) {Linear case ($F = L$): 
            \begin{equation*}F_{d,m}^\e = \tilde{F}_{d,m}^\e = L
            \end{equation*}};
     \end{scope}
     
 \node[punktchain, join] (non-osc) {Homogenization of  
     \begin{equation*}
     (\partial_t - \tilde{F}_{d,m}^\e)  = 0
     \end{equation*}
     with non-oscillatory \\initial data
     \begin{equation*}
     \bar{g}_{d,m}^\e (\cdot)
     \end{equation*}};
     
\node[punktchain, join] (int) {Interior correction
     \begin{equation*}
     \begin{split}
     \tilde{u}_{d,m}^\e &= \tilde{w}_{d,m}^\e + \e^2 u_{d+1,m}^\e \\
     &\quad + O (\e^{m-1}) 
     \end{split}
     \end{equation*}};
     
\draw[|-,-|,->, thick,] (int.west) |-+(0,0em)-| (scale.south);

\draw[-|,|-,->, thick, dotted] (linear.south) -|+(0,0em)|- (non-osc.east);

\draw[-|,|-,->, thick,] (scale.north) -|+(0,0em)|- (intro.west);

\draw[-|,|-,->, thick, dotted] (linear.north) -|+(0,0em)|- (intro.east);
\end{tikzpicture}\caption{Iteration Scheme at $d$-th Step} \label{figure:iter}
\end{figure}

Let us begin with the higher order correction near the initial layer.

\begin{lemma}\label{lemma:inid-1}
Under the assumption of Theorem \ref{theorem:higher}, one can construct $\{v_k\}_{k=0}^\infty$, $\{\bar{v}_k\}_{k=0}^\infty$, $\tilde{v}_m^\e$ and $\bar{g}_m^\e$ as in Proposition \ref{proposition:ini}. Let $\tilde{u}_m^\e$ be the bounded viscosity solution of 
\begin{equation}\label{eq:utme-pde}
\begin{dcases}
\begin{aligned}
\partial_t \tilde{u}_m^\e  &= \frac{1}{\e^2} F \left( \e^2 D^2 \tilde{u}_m^\e + \sum_{k=0}^m \e^k V_k ,x,t,\frac{x}{\e},\frac{t}{\e^2}\right) \\
& \quad - \frac{1}{\e^2} F \left( \sum_{k=0}^m \e^k V_k ,x,t,\frac{x}{\e},\frac{t}{\e^2}\right)
\end{aligned}
& \text{in }\R^n\times(0,T),\\
\tilde{u}_m^\e (x,0) = \bar{g}_m^\e (x) & \text{on }\R^n,
\end{dcases}
\end{equation} 
with $V_k$ given as in \eqref{eq:Vtk}. Then one has, for any $0<\e\leq\frac{1}{2}$, 
\begin{equation}\label{eq:inid-1}
\left| u^\e (x,t) - \tilde{v}_m^\e (x,t) - \tilde{u}_m^\e (x,t) \right| \leq C_m \e^{m-1},
\end{equation}
for all $x\in\R^n$ and $0\leq t\leq T$. In particular, 
\begin{equation}\label{eq:inid-log-1}
\left| u^\e (x,t) - \tilde{u}_m^\e (x,t) \right| \leq C_m \e^{m-1},
\end{equation}
for all $x\in\R^n$ and $c_m \e^2 |\log \e| \leq t\leq T$. 
\end{lemma}

\begin{proof}
Let us first prove the lemma for linear operators and then consider fully nonlinear ones.

\begin{case} Linear problems.
\end{case}

This lemma is straightforward when $F$ is a linear functional. Say $F (P,x,t,y,s) = \tr(A(x,t,y,s)P)$, and let $\tilde{u}_m^\e$ be the solution of the linear version of \eqref{eq:utme-pde}; i.e., 
\begin{equation}\label{eq:utme-pde-lin}
\begin{dcases}
\partial_t \tilde{u}_m^\e  = \tr \left(A \left( x,t,\frac{x}{\e},\frac{t}{\e^2}\right)D^2 \tilde{u}_m^\e \right) & \text{in }\R^n\times(0,T),\\
\tilde{u}_m^\e (x,0) = \bar{g}_m^\e (x) & \text{on }\R^n.
\end{dcases}
\end{equation}
In view of the initial value problem \eqref{eq:vtme-pde-lin} for $\tilde{v}_m^\e$ for the linear case, we have 
\begin{equation*}
\left( \tilde{u}_m^\e + \tilde{v}_m^\e \right) = \tr \left(A \left( x,t,\frac{x}{\e},\frac{t}{\e^2}\right)D^2 \left( \tilde{u}_m^\e + \tilde{v}_m^\e \right) \right) + \psi_m^\e \left(x,t,\frac{x}{\e},\frac{t}{\e^2}\right) \quad \text{in }\R^n\times(0,T),
\end{equation*}
where $\psi_m^\e$ is given as in Proposition \ref{proposition:ini}, and 
\begin{equation*}
\tilde{u}_m^\e(x,0) + \tilde{v}_m^\e (x,0) = g \left(x,\frac{x}{\e}\right) \quad \text{on }\R^n. 
\end{equation*} 
In comparison with the initial value problem \eqref{eq:ue-pde} for $u^\e$, one can deduce from the standard comparison principle \cite{CIL} that
\begin{equation*}
\left| u^\e (x,t) - \tilde{v}_m^\e (x,t) - \tilde{u}_m^\e (x,t) \right| \leq \sup_{\xi \in \R^n} \int_0^t \left| \psi_m^\e (\xi,\tau,\e^{-1} \xi,\e^{-2} \tau ) \right| d\tau \leq C_m \e^{m-1},
\end{equation*} 
for any $(x,t)\in\R^n\times[0,T]$, where the last inequality follows easily from the exponential decay estimate \eqref{eq:psime-decay} of $\psi_m^\e$. Moreover, we obtain \eqref{eq:inid-log-1} from the exponential decay estimate of $\tilde{v}_m^\e$ (due to $v_k - \bar{v}_k \in E(2,\bar\alpha;k)$ for any $k\geq 0$).

\begin{case}
Nonlinear problems.
\end{case}

The basic idea is the same as in the linear case. Here the nonlinear coupling effect adds more error terms in the comparison principle, which will be the main focus below.

We claim that $\tilde{u}_m^\e + \tilde{v}_m^\e$ (with $\tilde{u}_m^\e$ a solution to \eqref{eq:utme-pde}) solves 
\begin{equation}\label{eq:utme-vtme-pde}
\begin{dcases}
\begin{aligned}
\p_t \left( \tilde{u}_m^\e + \tilde{v}_m^\e \right) &= \frac{1}{\e^2} F \left( \e^2 D^2 \left( \tilde{u}_m^\e + \tilde{v}_m^\e\right), x, t, \frac{x}{\e}, \frac{t}{\e^2} \right) \\
&\quad + r_m^\e \left(x,t,\frac{x}{\e},\frac{t}{\e^2}\right) 
\end{aligned}
& \text{in }\R^n\times(0,T),\\
\left( \tilde{u}_m^\e + \tilde{v}_m^\e \right) (x,0) = g\left(x,\frac{x}{\e}\right) & \text{on }\R^n,
\end{dcases}
\end{equation}
with the remainder term $r_m^\e$ satisfying
\begin{equation}\label{eq:rme-decay}
\left |r_m^\e (x,t,y,s) \right| \leq C_m \e^{m-1} e^{-\beta_m s}.
\end{equation}
The rest of the proof then follows the same argument as in the linear case above. 

From the initial conditions of \eqref{eq:utme-pde} and \eqref{eq:vtme-pde}, one can easily verify the initial condition of \eqref{eq:utme-vtme-pde}. Hence, it only remains to check the interior equation and the exponential decay estimate of the remainder term. However, from the computation \eqref{eq:vtme-x} of $D^2\tilde{v}_m^\e$, one can proceed as 
\begin{equation*}
\begin{split}
& F \left( \e^2 D^2 \left( \tilde{u}_m^\e + \tilde{v}_m^\e \right),x,t,\frac{x}{\e},\frac{t}{\e^2}\right) \\
& = F \left( \e^2 D^2 \tilde{u}_m^\e + \sum_{k=0}^m \e^k V_k + \sum_{k=m+1}^{m+2} \e^k \tilde{V}_k,x,t,\frac{x}{\e},\frac{t}{\e^2}\right) \\
& = F \left( \e^2 D^2 \tilde{u}_m^\e + \sum_{k=0}^m \e^k V_k,x,t,\frac{x}{\e},\frac{t}{\e^2}\right) + h_m^{1,\e} \left(x,t,\frac{x}{\e},\frac{t}{\e^2}\right) \\
& = \e^2 \p_t \tilde{u}_m^\e + F \left( \sum_{k=0}^m \e^k V_k,x,t,\frac{x}{\e},\frac{t}{\e^2} \right) + h_m^{1,\e} \left(x,t,\frac{x}{\e},\frac{t}{\e^2}\right) \\
& = \e^2 \p_t \tilde{u}_m^\e + F \left( \sum_{k=0}^m \e^k V_k + \sum_{k=m+1}^{m+2} \e^k \tilde{V}_k,x,t,\frac{x}{\e},\frac{t}{\e^2} \right) + \left( h_m^{1,\e} + h_m^{2,\e}\right) \left(x,t,\frac{x}{\e},\frac{t}{\e^2}\right) \\
& = \e^2 \p_t \left( \tilde{u}_m^\e + \tilde{v}_m^\e \right) + \left( - \e^2 \psi_m^\e + h_m^{1,\e} + h_m^{2,\e}\right) \left(x,t,\frac{x}{\e},\frac{t}{\e^2}\right),
\end{split}
\end{equation*}
where $\tilde{V}_k$ is given by \eqref{eq:Vtk}, $\psi_m^\e$ is given as in Proposition \ref{proposition:ini} and by $h_m^{1,\e}$ and $h_m^{2,\e}$ we simply denoted the terms so that we have the equalities above. Let us remark that $h_m^{1,\e}$ and $h_m^{2,\e}$ are well-defined, since $D^2 \tilde{u}_m^\e$ and $\partial_t \tilde{u}_m^\e$ exist in the classical sense. This is because the operator governing the interior equation \eqref{eq:utme-pde} for $\tilde{u}_m^\e$ is uniformly elliptic, smooth and concave; here the smoothness comes from the fact that $V_k \in E(0,\bar\alpha;k)$. Hence, the standard regularity theory \cite{W2} ensures the smoothness of $\tilde{u}_m^\e$, although it may not possess a uniform regularity for the time derivative and the spatial Hessian. In addition, it follows from the ellipticity condition \eqref{eq:ellip-F} of $F$ that for each $i=1,2$, one has 
\begin{equation}\label{eq:hmie-decay}
\left| h_m^{i,\e} (x,t,y,s) \right| \leq C_0 \sum_{k=m+1}^{m+2} \e^k \left| \tilde{V}_k(x,t,y,s)\right| \leq C_m \e^{m+1} e^{-\beta_m s},
\end{equation}
for any $0<\e\leq\frac{1}{2}$, where the second estimate follows from the exponential decay estimate of $v_k - \bar{v}_k$. 

This shows that $\tilde{u}_m^\e + \tilde{v}_m^\e$ satisfies the interior equation of \eqref{eq:utme-vtme-pde} with 
\begin{equation*}
r_m^\e = \psi_m^\e - \e^{-2} \left( h_m^{1,\e} + h_m^{2,\e} \right).
\end{equation*}
The decay estimate \eqref{eq:rme-decay} of $r_m^\e$ can be deduced from \eqref{eq:psime-decay} and \eqref{eq:hmie-decay}, which finishes the proof. 
\end{proof} 

In view of \eqref{eq:utme-pde}, one may realize that we are in a position to invoke Proposition \ref{proposition:int} to construct the higher order interior correctors for the new homogenization problem. This eventually leads us to a higher order approximation of $\tilde{u}_m^\e$ by the interior correctors, again up to some function with order $\e^2$. The function turns out to be a viscosity solution to a new homogenization problem essentially belongs to the same class of \eqref{eq:ue-pde}, which allows us to run a bootstrap argument. 

\begin{lemma}\label{lemma:intd-1}
Under the conclusion of Lemma \ref{lemma:inid-1}, let $\{V_k\}_{k=0}^\infty$ and $\{\bar{g}_k\}_{k=0}^\infty$ be as in \eqref{eq:Vk} and respectively \eqref{eq:gbk}. Then one can construct $\{w_k\}_{k=0}^\infty$, $\{w_k^\#\}_{k=0}^\infty$, $\{\bar{u}_k\}_{k=0}^\infty$, $\tilde{w}_m^\e$ and $\tilde{w}_m^{\#,\e}$ as in Proposition \ref{proposition:int}. Let $u_{1,m}^\e$ be the bounded viscosity solution to  
\begin{equation}\label{eq:u1me-pde}
\begin{dcases}
\begin{aligned}
\p_t u_{1,m}^\e &= \frac{1}{\e^4} F \left( \e^4 D^2 u_{1,m}^\e + \sum_{k=0}^m \e^k (V_k + W_k),x,t,\frac{x}{\e},\frac{t}{\e^2}\right) \\
& \quad - \frac{1}{\e^4} F \left( \sum_{k=0}^m \e^k (V_k + W_k) ,x,t,\frac{x}{\e},\frac{t}{\e^2}\right)
\end{aligned}
& \text{in }\R^n\times(0,T),\\
u_{1,m}^\e (x,0) = - \sum_{k=0}^{m-2} \e^k  w_{k+2} \left(x,0,\frac{x}{\e},0\right) & \text{on }\R^n,
\end{dcases}
\end{equation}
where $W_k$ is given as in \eqref{eq:Wk}. Then for any $ 0< \e\leq\frac{1}{2}$,  
\begin{equation}\label{eq:intd-1}
\left| \tilde{u}_m^\e (x,t) - \tilde{w}_m^\e (x,t) - \e^2 u_{1,m}^\e (x,t) \right| \leq C_m \e^{m-1},
\end{equation} 
for all $x\in\R^n$ and $0\leq t\leq T$. Moreover, one has  
\begin{equation}\label{eq:intd-log-1}
\left| \tilde{u}_m^\e (x,t) - \tilde{w}_m^{\#,\e} (x,t) - \e^2 u_{1,m}^\e (x,t) \right| \leq C_m \e^{m-1}
\end{equation} 
for all $x\in\R^n$ and $c_m\e^2 |\log \e|\leq t\leq T$.
\end{lemma}

\begin{proof}

As in the previous lemma, let us divide the proof into two cases.

\begin{case} Linear problems.
\end{case}

As in the proof of Lemma \ref{lemma:inid-1}, we shall begin with the linear case, and assume that $F(P,x,t,y,s) = \tr(A(x,t,y,s)P)$. Let $u_{1,m}^\e$ be the bounded solution to the linear version of \eqref{eq:u1me-pde}, i.e.,
\begin{equation}\label{eq:u1me-pde-lin}
\begin{dcases}
\partial_t  u_{1,m}^\e  = \tr \left(A \left( x,t,\frac{x}{\e},\frac{t}{\e^2}\right)D^2 u_{1,m}^\e \right) & \text{in }\R^n\times(0,T),\\
u_{1,m}^\e (x,0) = - \sum_{k=0}^{m-2} \e^k w_{k+2} \left(x,0,\frac{x}{\e},0\right) & \text{on }\R^n.
\end{dcases}
\end{equation}
Recall from Remark \ref{remark:int2} and Remark \ref{remark:int5} that $w_k = w_k^\#$ when $F$ is linear. 

In view of the equation \eqref{eq:wtme-pde-lin} for $\tilde{w}_m^\e$, we see that
\begin{equation*}
\partial_t \left( \tilde{w}_m^\e + \e^2 u_{1,m}^\e \right) = \tr \left(A \left( x,t,\frac{x}{\e},\frac{t}{\e^2}\right)D^2 \left( \tilde{w}_m^\e + \e^2 u_{1,m}^\e \right)  \right) + \psi_m^\e \left(x,t,\frac{x}{\e},\frac{t}{\e^2}\right) \quad\text{in }\R^n\times(0,T),
\end{equation*} 
while 
\begin{equation*}
\begin{split}
\tilde{w}_m^\e(x,0)  + \e^2 u_{1,m}^\e (x,0) &= \sum_{k=0}^m \e^k \left( w_k \left(x,0\frac{x}{\e},0\right) + \bar{u}_k(x,0)\right) - \e^2 \sum_{k=0}^{m-2} \e^k  w_{k+2} \left(x,0,\frac{x}{\e},0\right) \\
& = \sum_{k=0}^m \e^k \bar{u}_k (x,0) =  \bar{g}_m^\e(x),
\end{split}
\end{equation*} 
where we used the fact that $w_0 = w_1 = 0$ and $\bar{u}_k(\cdot,0) = \bar{g}_k$. Hence, we can compare $\tilde{w}_m^\e + \e^2 u_{1,m}^\e$ with $\tilde{u}_m^\e$, which solves \eqref{eq:utme-pde-lin}. This implies the first estimate \eqref{eq:intd-1}. Since we have $w_k = w_k^\#$ in the linear case (see Remark \ref{remark:int5}), we also have $\tilde{w}_m^\e = \tilde{w}_m^{\#,\e}$, so the second estimate \eqref{eq:intd-log-1} becomes trivial from \eqref{eq:intd-1}. 

\begin{case}
Nonlinear problems.
\end{case}

We assert that $\tilde{w}_m^\e + \e^2 u_{1,m}^\e$ is a viscosity solution to 
\begin{equation}\label{eq:wtme-u1me-pde}
\begin{dcases}
\begin{aligned}
\partial_t \left( \tilde{w}_m^\e + \e^2 u_{1,m}^\e \right) & = \frac{1}{\e^2} F \left( \e^2 D^2 \left( \tilde{w}_m^\e + \e^2 u_{1,m}^\e \right) + \sum_{k=0}^m \e^k V_k,x,t,\frac{x}{\e},\frac{t}{\e^2}\right) \\
& \quad - \frac{1}{\e^2} F \left( \sum_{k=0}^m \e^k V_k,x,t,\frac{x}{\e},\frac{t}{\e^2}\right) + r_m^\e \left(x,t,\frac{x}{\e},\frac{t}{\e^2}\right)
\end{aligned}
& \text{in }\R^n\times(0,T),\\
\left(\tilde{w}_m^\e + \e^2 u_{1,m}^\e \right)(x,0) = \bar{g}_m^\e(x) & \text{on }\R^n, 
\end{dcases}
\end{equation}
with some remainder term $r_m^\e$ satisfying
\begin{equation}\label{eq:rme-Linf}
\left| r_m^\e (x,t,y,s) \right| \leq C_m \e^{m-1}.
\end{equation}
Then one can deduce the desired estimate \eqref{eq:intd-1} by means of the comparison principle, as in the proof of Lemma \ref{lemma:inid-1}. Moreover the error estimate \eqref{eq:intd-log-1} away from the initial time layer follows from \eqref{eq:intd-1} and the exponential decay estimate of $w_k - w_k^\#$. 

Note that $V_k\in E(0,\bar\alpha;k)$ and that $\bar{v}_k \in S(k)$, which implies $\bar{g}_k = \bar{v}_k (\cdot,0) \in S(k)$, for any $k\geq 0$, so the sequences $\{V_k\}_{k=0}^\infty$ and $\{\bar{g}_k\}_{k=0}^\infty$ satisfy the assumption of Proposition \ref{proposition:int}. Thus, we obtain the sequence $\{w_k\}_{k=0}^\infty$ of higher order interior correctors, and the sequence $\{\bar{u}_k\}_{k=0}^\infty$ of higher order effective limits. 

The initial condition for $\tilde{w}_m^\e + \e^2 u_{1,m}^\e$ can be proved similarly as in the linear case, so we shall only focus on  the interior equation it satisfies. Since we have 
\begin{equation*}
\e^2 D^2 \tilde{w}_m^\e (x,t) = \sum_{k=0}^{m+2} \e^k \tilde{W}_k \left(x,t,\frac{x}{\e},\frac{t}{\e^2}\right) = \sum_{k=0}^m \e^k W_k \left(x,t,\frac{x}{\e},\frac{t}{\e^2}\right) + \sum_{k=m+1}^{m+2} \e^k \tilde{W}_k \left(x,t,\frac{x}{\e},\frac{t}{\e^2}\right),
\end{equation*}
with $\tilde{W}_k$ given as in \eqref{eq:Wtk}, it follows from the interior equations of \eqref{eq:wtme-pde} that 
\begin{equation*}
\begin{split}
& F \left( \e^2 D^2 \left( \tilde{w}_m^\e + \e^2 u_{1,m}^\e\right) + \sum_{k=0}^m \e^k V_k,x,t,\frac{x}{\e},\frac{t}{\e^2}\right) \\
& = F \left( \e^4 D^2 u_{1,m}^\e + \sum_{k=0}^m \e^k (V_k + W_k) + \sum_{k=m+1}^{m+2} \e^k \tilde{W}_k,x,t,\frac{x}{\e},\frac{t}{\e^2}\right) \\
& = F \left( \e^4 D^2 u_{1,m}^\e + \sum_{k=0}^m \e^k (V_k + W_k),x,t,\frac{x}{\e},\frac{t}{\e^2}\right) + h_m^{1,\e} \left(x,t,\frac{x}{\e},\frac{t}{\e^2}\right)  \\
& = \e^4 \partial_t u_{1,m}^\e + F \left( \sum_{k=0}^m \e^k (V_k + W_k) ,x,t,\frac{x}{\e},\frac{t}{\e^2}\right) + h_m^{1,\e} \left(x,t,\frac{x}{\e},\frac{t}{\e^2}\right) \\
& =  \e^4 \partial_t u_{1,m}^\e + F \left( \sum_{k=0}^m \e^k (V_k + W_k) + \sum_{k=m+1}^{m+2} \e^k \tilde{W}_k ,x,t,\frac{x}{\e},\frac{t}{\e^2}\right) + \left(h_m^{1,\e} + h_m^{2,\e}\right) \left(x,t,\frac{x}{\e},\frac{t}{\e^2}\right) \\
& = \e^2 \partial_t \left( \tilde{w}_m^\e + \e^2 u_{1,m}^\e\right) + \left( -\e^2 \psi_m^\e +  h_m^{1,\e} + h_m^{2,\e} \right) \left(x,t,\frac{x}{\e},\frac{t}{\e^2}\right),
\end{split}
\end{equation*}
where $\psi_m^\e$ is given by \eqref{eq:wtme-pde}, and by $h_m^{1,\e}$ and $h_m^{2,\e}$ we simply denoted the terms so that we have the equalities above. Arguing similarly as with the proof of Lemma \ref{lemma:inid-1}, one can justify the well-definedness of $h_m^{i,\e}$, for $i=1,2$, and deduce that 
\begin{equation}\label{eq:hmie-Linf}
\left| h_m^{i,\e} (x,t,y,s) \right| \leq C_0 \sum_{k=m+1}^{m+2} \e^k \left| \tilde{W}_k(x,t,y,s)\right| \leq C_m \e^{m+1},
\end{equation}
for any $0<\e\leq\frac{1}{2}$, where the second estimate follows from the observation that $w_k \in S(2,\bar\alpha;k)$. 

Hence, $\tilde{w}_m^\e + \e^2 u_{1,m}^\e$ satisfies the interior equation of \eqref{eq:wtme-u1me-pde} with 
\begin{equation*}
r_m^\e  = \psi_m^\e - \e^{-2} \left(h_m^{1,\e} + h_m^{2,\e}\right),
\end{equation*}
and the estimate \eqref{eq:rme-Linf} of $r_m^\e$ follows from \eqref{eq:psime-psimhe-decay-Linf} and \eqref{eq:hmie-Linf}. This finishes the proof. 
\end{proof}

As a corollary to the above lemmas, we achieve the following.

\begin{corollary}\label{corollary:iniintd-1}
One has, for any $0<\e\leq\frac{1}{2}$,  
\begin{equation}\label{eq:iniintd-1}
\left| u^\e (x,t) - \tilde{v}_m^\e(x,t) - \tilde{w}_m^\e (x,t) - \e^2 u_{1,m}^\e (x,t) \right| \leq C_m \e^{m-1},
\end{equation}
for all $x\in\R^n$ and $0\leq t\leq T$. In addition, 
\begin{equation}\label{eq:iniintd-log-1}
\left| u^\e (x,t) - \tilde{w}_m^{\#,\e} (x,t) - \e^2 u_{1,m}^\e (x,t) \right| \leq C_m \e^{m-1},
\end{equation} 
for all $x\in\R^n$ and $c_m\e^2 |\log \e| \leq t\leq T$. 
\end{corollary}

It is worthwhile to repeat that $u_{1,m}^\e$ is a solution to a homogenization problem essentially of the same type with \eqref{eq:ue-pde}. Hence, we can iterate the above arguments, provided that we can construct the higher order initial layer correctors and interior correctors in a more general setting. Here we shall only present the argument and skip the proof, since the main idea and the computations are already shown in the proofs of Proposition \ref{proposition:ini} and Proposition \ref{proposition:int}.

\begin{remark}\label{remark:higher-100}
One can finish the proof of our main theorem at this point in the case of linear operators. We also find it more helpful to the reader, since the proof for nonlinear operators involves many new notations and definitions, so it could be confusing without keeping the basic structure of the proof in mind. The argument goes as follows:

One may have already observed that the initial data in \eqref{eq:u1me-pde-lin} for $u_{1,m}^\e$ is slightly more general than the original problem \eqref{eq:ue-pde}, as the former is a (finite) summation of rapidly oscillating functions with different orders of $\e$. This motivates us to consider the higher order correction of the following problem,
\begin{equation}\label{eq:udme-pde-lin}
\begin{dcases}
\partial_t u_{d,m}^\e = \tr \left( A\left(x,t,\frac{x}{\e},\frac{t}{\e^2}\right) D^2 u_{d,m}^\e \right) & \text{in }\R^n\times(0,T),\\
u_{d,m}^\e (x,0) = \sum_{k=0}^{m-2d} \e^k g_{d,k}\left(x,\frac{x}{\e}\right) & \text{on }\R^n.\\
\end{dcases}
\end{equation}
Note that the original problem \eqref{eq:ue-pde} is the case when $d=0$, $g_{0,0} = g$ and $g_{0,k} = 0$ for any $k\geq 1$, whereas the newly generated problem \eqref{eq:u1me-pde-lin} is the case when $d=1$, $g_{1,k} = - w_{k+2}(\cdot,0,\cdot,0)$ for any $k\geq 0$. 

Therefore, applying one more round of higher order correction near the initial layer (Lemma \ref{lemma:inid-1}) and in the interior (Lemma \ref{lemma:intd-1}), we obtain 
\begin{equation}\label{eq:higher-lin}
\left| u_{d,m}^\e (x,t) - \tilde{v}_{d,m}^\e (x,t) - \tilde{w}_{d,m}^\e (x,t) - \e^2 u_{d+1,m}^\e (x,t) \right| \leq C_{d,m}\e^{m-1},
\end{equation} 
for any $x\in\R^n$ and any $0\leq t\leq T$, where $\tilde{v}_{d,m}^\e$ and $\tilde{w}_{d,m}^\e$ are the truncated summations of the higher order initial layer correctors and, respectively, the higher order interior correctors (at $d$-th step), while $u_{d+1,m}^\e$ is the bounded solution of \eqref{eq:udme-pde-lin} with the index $d$ replaced by $d+1$; for a more precise definition of $\tilde{v}_{d,m}^\e$ and $\tilde{w}_{d,m}^\e$, see Proposition \ref{proposition:inid} and Proposition \ref{proposition:intd}.

Beginning with $d=1$, one can iterate this argument for $\lfloor \frac{m}{2} \rfloor-1$ times. Adding the resulting estimates \eqref{eq:higher-lin} for $1\leq d\leq \lfloor \frac{m}{2} \rfloor-1$, and combining the resultant with the previous estimate \eqref{eq:iniintd-1}, we arrive at
\begin{equation*}
\left| u^\e (x,t) - \sum_{d=0}^{\lfloor \frac{m}{2} \rfloor -1} \e^d \left(\tilde{v}_{d,m}^\e (x,t) + \tilde{w}_{d,m}^\e(x,t)\right) + \e^{2\lfloor \frac{m}{2} \rfloor} u_{\lfloor \frac{m}{2} \rfloor}(x,t) \right| \leq C_m \e^{m-1},
\end{equation*}
for any $x\in\R^n$ and any $0\leq t\leq T$. Since $\e^{2\lfloor \frac{m}{2} \rfloor}  \leq \e^{m-1}$ and $\norm{u_{\lfloor \frac{m}{2} \rfloor}}_{L^\infty(\R^n\times[0,T])}\leq C_m$, this finishes the proof of \eqref{eq:higher} (as well as \eqref{eq:higher-log}) in the case of linear operators. 
\end{remark}

First comes the construction of higher order initial layer correctors.

\begin{proposition}\label{proposition:inid}
Assume that $F$ verifies \eqref{eq:ellip-F} - \eqref{eq:Ck-Ca-F}. Fix integers $d\geq 0$ and $m\geq 2d$. Let $\{X_{d,k} : \R^n\times[0,T]\times\T^n\times[0,\infty)\ra\cS^n\}_{k=0}^\infty$, $\{X_{d,k}^\# : \R^n\times[0,T]\times\T^n\times\T \ra \cS^n\}_{k=0}^\infty$ and $\{g_{d,k}: \R^n\times\T^n\ra \R\}_{k=0}^\infty$ be such that $X_{d,k}, X_{d,k}^\# \in S(0,\bar\alpha;d,k)$,  $X_{d,k} - X_{d,k}^\#\in E(0,\bar\alpha;d,k)$ and $g_{d,k} \in S(2,\bar\alpha;d,k)$. 

Then there exist $\{v_{d,k} : \R^n\times[0,T]\times\T^n\times[0,\infty)\ra \R\}_{k=0}^\infty$, $\{\bar{v}_{d,k}: \R^n\times[0,T]\ra \R\}_{k=0}^\infty$ such that $v_{d,k} \in S(2,\bar\alpha;d,k)$, $\bar{v}_{d,k} \in S(d,k)$ and $v_{d,k} - \bar{v}_{d,k} \in E(2,\bar\alpha;d,k)$ and the following is true. Set 
\begin{equation*}
\tilde{v}_{d,m}^\e (x,t) = \sum_{k=0}^{m-2d} \e^k \left( v_{d,k} \left(x,t,\frac{x}{\e},\frac{t}{\e^2}\right) - \bar{v}_{d,k}(x,t)\right), \quad  \bar{g}_{d,m}^\e (x) = \sum_{k=0}^{m-2d} \e^k \bar{v}_{d,k}(x,0).
\end{equation*}
Then $\tilde{v}_{d,m}^\e$ and $\bar{g}_{d,m}^\e$ satisfy 
\begin{equation*}
\begin{dcases}
\!\begin{aligned}
\p_t \tilde{v}_{d,m}^\e &= \frac{1}{\e^{2d+2}} F \left( \e^{2d+2} D^2 \tilde{v}_{d,m}^\e + \sum_{k=0}^m \e^k X_{d,k}  ,x,t,\frac{x}{\e},\frac{t}{\e^2}\right) \\
&\quad -  \frac{1}{\e^{2d+2}} F \left( \sum_{k=0}^m \e^k X_{d,k} ,x,t,\frac{x}{\e},\frac{t}{\e^2}\right) + \psi_{d,m}^\e \left(x,t,\frac{x}{\e},\frac{t}{\e^2}\right) 
\end{aligned}
& \text{in }\R^n\times(0,T),\\
\tilde{v}_{d,m}^\e (x,0) + \bar{g}_{d,m}^\e(x) = \sum_{k=0}^{m-2d} \e^k g_{d,k} \left(x,\frac{x}{\e}\right) & \text{on }\R^n,
\end{dcases}
\end{equation*}
with some $\psi_{d,m}^\e:\R^n\times[0,T]\times\T^n\times[0,\infty)\ra\R$ verifying 
\begin{equation*}
\left| \psi_{d,m}^\e (x,t,y,s) \right| \leq C_{d,m} \e^{m-2d-1} e^{-\beta_{d,m}s},
\end{equation*}
for any $0<\e\leq\frac{1}{2}$. 
\end{proposition}

\begin{remark}\label{remark:inid}
One may notice that Proposition \ref{proposition:ini} is simply the special case with $d = 0$, and $X_{0,k} = X_{d,k}^\# = 0$ for any $k \geq 0$, $g_{0,0} = g$ and $g_{0,k} = 0$ for any $k\geq 1$. Moreover, the new homogenization problem \eqref{eq:u1me-pde} falls under the case $d=1$, $X_{1,k} = V_k + W_k$, $X_{1,k}^\# = W_k^\#$ and $g_{1,k} (x,y) = \bar{u}_{k+2}(x,0) - w_{k+2} (x,0,y,0)$. 
\end{remark}

Next follows the construction of the higher order interior correctors. 

\begin{proposition}\label{proposition:intd} 
Assume that $F$ verifies \eqref{eq:ellip-F} - \eqref{eq:Ck-Ca-F}. Fix integers $d\geq 0$ and $m\geq 2d$. Let $\{Y_{d,k} : \R^n\times[0,T]\times\T^n\times[0,\infty)\ra\cS^n\}_{k=0}^\infty$, $\{Y_{d,k}^\# \in \R^n\times[0,T]\times\T^n\times\T\ra\cS^n\}_{k=0}^\infty$ and $\{\bar{g}_{d,k} : \R^n\ra\R\}_{k=0}^\infty$ be such that $Y_{d,k}, Y_{d,k} \in S(0,\bar\alpha;d,k)$, $Y_{d,k} - Y_{d,k}^\#\in E(2,\bar\alpha;d,k)$ and $\bar{g}_{d,k} \in S(d,k)$. 

Then there are $\{w_{d,k} : \R^n\times[0,T]\times\T^n\times[0,\infty) \ra \R\}_{k=0}^\infty$, $\{w_{d,k}^\# : \R^n\times[0,T]\times\T^n\times\T \ra \R\}_{k=0}^\infty$ and $\{\bar{u}_{d,k} : \R^n\times[0,T]\ra\R\}_{k=0}^\infty$ satisfying $w_{d,k}, w_{d,k} \in S(2,\bar\alpha;d,k)$, $w_{d,k} - w_{d,k}^\# \in E(2,\bar\alpha;d,k)$, $w_{d,0} = w_{d,1} = w_{d,0}^\# = w_{d,1}^\# = 0$, as well as $\bar{u}_{d,k} \in S(d,k)$ and $\bar{u}_{d,k}(\cdot,0) = \bar{g}_{d,k}$ for any $k\geq 0$, such that the following is true. Define 
\begin{align*}
\tilde{w}_{d,m}^\e (x,t) &= \sum_{k=0}^{m-2d} \e^k \left( w_{d,k} \left( x,t,\frac{x}{\e},\frac{t}{\e^2}\right) + \bar{u}_{d,k}(x,t)\right),\\
\tilde{w}_{d,m}^{\#,\e} (x,t) &= \sum_{k=0}^{m-2d} \e^k \left( w_{d,k}^\# \left( x,t,\frac{x}{\e},\frac{t}{\e^2}\right) + \bar{u}_{d,k}(x,t)\right).
\end{align*}
Then one has 
\begin{equation*}
\begin{split}
\p_t \tilde{w}_{d,m}^\e &= \frac{1}{\e^{2d+2}} F \left( \e^{2d+2} D^2\tilde{w}_m^\e + \sum_{k=0}^m \e^k Y_{d,k}, x,t,\frac{x}{\e},\frac{t}{\e^2}\right) \\
&\quad - \frac{1}{\e^{2d+2}} F \left( \sum_{k=0}^m \e^k Y_{d,k},x,t,\frac{x}{\e},\frac{t}{\e^2}\right) + \psi_{d,m}^\e \left(x,t,\frac{x}{\e},\frac{t}{\e^2}\right) \quad \text{in }\R^n\times(0,T),
\end{split}
\end{equation*}
and 
\begin{equation*}
\begin{split}
\p_t \tilde{w}_{d,m}^{\#,\e} &= \frac{1}{\e^{2d+2}} F \left( \e^{2d+2} D^2\tilde{w}_m^{\#,\e} + \sum_{k=0}^m \e^k Y_{d,k}^\#, x,t,\frac{x}{\e},\frac{t}{\e^2}\right) \\
&\quad - \frac{1}{\e^{2d+2}} F \left( \sum_{k=0}^m \e^k Y_{d,k}^\#,x,t,\frac{x}{\e},\frac{t}{\e^2}\right) + \psi_{d,m}^{\e,\#} \left(x,t,\frac{x}{\e},\frac{t}{\e^2}\right) \quad \text{in }\R^n\times(0,T),
\end{split}
\end{equation*}
with some $\psi_{d,m}^\e: \R^n\times[0,T]\times\T^n\times [0,\infty) \ra \R$ and  $\psi_{d,m}^{\#,\e}: \R^n\times[0,T]\times\T^n\times\T \ra \R$ satisfying 
\begin{equation*}
\left| \psi_{d,m}^{\#,\e} (x,t,y,s) \right| + e^{\beta_{d,m}s} \left| (  \psi_{d,m}^\e- \psi_{d,m}^{\#,\e} )(x,t,y,s)\right|  \leq C_{d,m}\e^{m-2d-1},
\end{equation*}
for any $0<\e\leq\frac{1}{2}$. 
\end{proposition}

\begin{remark}\label{remark:intd}
Proposition \ref{proposition:int} is the special case with $d=0$, $Y_{0,k} = V_k$ and $Y_{d,k}^\# = 0$. 
\end{remark}

Finally, we are ready to prove our main result.

\begin{proof}[Proof of Theorem \ref{theorem:higher}] 
Since the case $2\leq m\leq 3$ is already proved in Corollary \ref{corollary:iniintd-1}, we shall consider the case $m\geq 4$ only. As in the case of linear operators (Remark \ref{remark:higher-100}), the main goal is to prove that given suitable sequences $\{X_{d,k} : \R^n\times[0,T]\times\T^n\times[0,\infty)\ra\cS^n\}_{k=0}^\infty$ and $\{g_{d,k} : \R^n\times\T^n\ra\R\}_{k=0}^\infty$, where $X_{d,k}$ amounts to the nonlinear coupling effect and $g_{d,k}$ newly obtained rapidly oscillating initial data, the bounded solution $u_{d,m}^\e$ to the following initial value problem, 
\begin{equation}\label{eq:udme-pde}
\begin{dcases}
\begin{aligned}
\p_t u_{d,m}^\e &= \frac{1}{\e^{2d+2}} F \left( \e^{2d+2} D^2 u_{r,m}^\e + \sum_{k=0}^{m-2d} \e^k X_{d,k},x,t,\frac{x}{\e},\frac{t}{\e^2}\right) \\
&\quad - \frac{1}{\e^{2d+2}} F \left( \sum_{k=0}^{m-2d} \e^k X_{d,k},x,t,\frac{x}{\e},\frac{t}{\e^2}\right)
\end{aligned}
& \text{in }\R^n\times(0,T),\\
u_{d,m}^\e (x,0) = \sum_{k=0}^{m-2d} \e^k g_{d,k}\left(x,\frac{x}{\e}\right) & \text{on }\R^n,
\end{dcases}
\end{equation}
satisfies 
\begin{equation}\label{eq:iniintd}
\left| u_{d,m}^\e (x,t) - \tilde{v}_{d,m}^\e (x,t) - \tilde{w}_{d,m}^\e (x,t) - \e^2 u_{d+1,m}^\e (x,t) \right| \leq C_{d,m} \e^{m-1},
\end{equation}
for any $(x,t)\in\R^n\times[0,T]$, where $\tilde{v}_{d,m}^\e$ and $\tilde{w}_{d,m}^\e$ are the truncated summation of higher order correctors given as in Proposition \ref{proposition:inid} and respectively Proposition \ref{proposition:intd}. Here $u_{d+1,m}^\e$ is the bounded solution of \eqref{eq:udme-pde} with $d$ replaced by $d+1$ for some other $\{X_{d+1,k}\}_{k=0}^\infty$ and $\{g_{d+1,k}\}_{k=0}^\infty$, which satisfy the same properties as $\{X_{d,k}\}_{k=0}^\infty$ and respectively $\{g_{d,k}\}_{k=0}^\infty$. If we have \eqref{eq:iniintd}, then from the aforementioned propositions, we can also derive 
\begin{equation}\label{eq:iniintd-log}
\left| u_{d,m}^\e (x,t) -  \tilde{w}_{d,m}^{\#,\e} (x,t) - \e^2 u_{d+1,m}^\e (x,t) \right| \leq C_{d,m} \e^{m-1},
\end{equation}
for all $x\in\R^n$ and $c_{d,m}\e^2 |\log \e|\leq t\leq T$, where $\tilde{w}_{d,m}^{\#,\e}$ is given as in Proposition \ref{proposition:intd}. The rest of the proof is then the same as in Remark \ref{remark:higher-100}, so we can omit it. 

Henceforth, we shall focus on the construction of the sequences $\{X_{d,k}\}_{k=0}^\infty$ and $\{g_{d,k}\}_{k=0}^\infty$. In fact, in order to proceeds with the higher order correction, we shall also construct sequences $\{X_{d,k}^\# : \R^n\times[0,T]\times\T^n\times\T\ra\cS^n\}_{k=0}^\infty$ and $\{Y_{d,k}^\#:\R^n\times[0,T]\times\T^n\times\T\ra \R\}_{k=0}^\infty$ of the space-time periodic versions, as well as a sequence $\{\bar{g}_{d,k}: \R^n\ra \R\}_{k=0}^\infty$ of higher order effective initial data. 

For the initial case $d=1$, take 
\begin{equation*}
X_{1,k} = V_k + W_k,\quad X_{1,k}^\# = W_k^\#,
\end{equation*}
where $V_k$, $W_k$ and $W_k^\#$ are as in \eqref{eq:Vk}, \eqref{eq:Wk} and respectively \eqref{eq:Wkh}. Also define \begin{equation*}
g_{1,k} = - w_{k+2} (\cdot,0,\cdot,0),
\end{equation*}
with $w_k$ as in Lemma \ref{lemma:intd-1}. One can see from above that $\{X_{1,k}\}_{k=0}^\infty$, $\{X_{1,k}^\#\}_{k=0}^\infty$ and $\{g_{1,k}\}_{k=0}^\infty$ satisfy the assumption of Proposition \ref{proposition:inid} with $d = 1$. 

Now let $d\geq 1$ be any, and suppose that $\{X_{d,k}\}_{k=0}^\infty$, $\{X_{d,k}^\#\}_{k=0}^\infty$ and $\{g_{d,k}\}_{k=0}^\infty$ are already given as in Proposition \ref{proposition:inid}. Then we obtain $\{v_{d,k}\}_{k=0}^\infty$ and $\{\bar{v}_{d,k}\}_{k=0}^\infty$ from which one can define, as in \eqref{eq:Vk},
\begin{equation*}
V_{d,k} = 
\begin{dcases}
D_y^2 v_{d,0}, & k =0,\\
D_y^2 v_{d,1} + D_{xy} v_{d,0}, & k =1,\\
D_y^2 v_{d,k} + D_{xy} v_{d,k-1} + D_x^2 (v_{d,k-2} - \bar{v}_{d,k-2}), & k\geq 2,
\end{dcases}
\end{equation*} 
and 
\begin{equation*}
\bar{g}_{d,k} = \bar{v}_{d,k}(\cdot,0).
\end{equation*}
Then we set, for $k\geq 0$, 
\begin{equation*}
Y_{d,k} = 
\begin{dcases}
X_{d,k}, & 0\leq k\leq 2d-1, \\ 
X_{d,k} + V_{d,k-2d}, & k \geq 2d,
\end{dcases}
\end{equation*}
and
\begin{equation*}
Y_{d,k}^\# = X_{d,k}^\#.
\end{equation*}
Then from the assumptions that $X_{d,k},X_{d,k}^\#\in S(0,\bar\alpha;d,k)$ with $X_{d,k} - X_{d,k}^\# \in E(0,\bar\alpha;d,k)$ and the observation that $V_{d,k}\in E(0,\bar\alpha;d,k)$ it follows that $\{Y_{d,k}\}_{k=0}^\infty$, $\{Y_{d,k}^\#\}_{k=0}^\infty$ and $\{\bar{g}_{d,k}\}_{k=0}^\infty$ defined as above satisfy the conditions of Proposition \ref{proposition:intd}. Thus, one obtains $\{w_{d,k}\}_{k=0}^\infty$, $\{w_{d,k}^\#\}_{k=0}^\infty$ and $\{\bar{u}_{d,k}\}_{k=0}^\infty$ as in the proposition.

With such a choice of higher order interior correctors and effective limit profiles, we set
\begin{equation*}
W_{d,k} = 
\begin{dcases}
0, & k = 0,1, \\ 
D_y^2 w_{d,k} + D_{xy} w_{d,k-1} + D_x^2 (w_{d,k-2} + \bar{u}_{d,k-2}), & k \geq 2,
\end{dcases}
\end{equation*} 
and its time-periodic version by 
\begin{equation*}
W_{d,k}^\# = 
\begin{dcases}
0, & k = 0,1, \\ 
D_y^2 w_{d,k}^\# + D_{xy} w_{d,k-1}^\# + D_x^2 (w_{d,k-2}^\# + \bar{u}_{d,k-2}), & k \geq 2.
\end{dcases}
\end{equation*}

Now define, for $k\geq 0$, 
\begin{equation*}
X_{d+1,k} = 
\begin{dcases}
Y_{d,k}, & 0\leq k\leq 2d+1,\\
Y_{d,k} + W_{d,k-2d}, & k \geq 2d+2. 
\end{dcases}
\end{equation*}
and respectively the time-periodic version by 
\begin{equation*}
X_{d+1,k}^\# = 
\begin{dcases}
Y_{d,k}^\# = X_{d,k}^\#, & 0\leq k\leq 2d+1,\\
Y_{d,k}^\# + W_{d,k-2d}^\# = X_{d,k}^\# + W_{d,k-2d}^\#, & k \geq 2d+2,
\end{dcases}
\end{equation*}
as well as the new oscillatory initial data by
\begin{equation*}
g_{d+1,k} = - w_{d,k+2} (\cdot,0,\cdot,0). 
\end{equation*} 
By means of $w_{d,k}\in S(2,\bar\alpha;d,k)$, $w_{d,k}^\# \in S(2,\bar\alpha;d,k)$ with $w_{d,k} - w_{d,k}^\# \in E(2,\bar\alpha;d,k)$, $\bar{u}_{d,k} \in S(d,k)$ and the assumptions on $Y_{d,k}$ and $Y_{d,k}^\#$, one can verify that $\{X_{d+1,k}\}_{k=0}^\infty$, $\{X_{d+1,k}^\#\}_{k=0}^\infty$ and $\{g_{d+1,k}\}_{k=0}^\infty$ also satisfy the assumptions of Proposition \ref{proposition:inid}, which allows us to run an induction argument. 

To this end, given $m\geq 4$ and $1\leq d\leq \lfloor \frac{m}{2} \rfloor - 1$, we obtain $\tilde{v}_{d,m}^\e$, $\bar{g}_{d,m}^\e$ satisfying Proposition \ref{proposition:inid} (iii), and $\tilde{w}_{d,m}^\e$, $\tilde{w}_{d,m}^{\#,\e}$ satisfying Proposition \ref{proposition:intd} (iv). Following the arguments from Lemma \ref{lemma:inid-1} to Lemma \ref{lemma:intd-1}, one can prove \eqref{eq:iniintd} as in the conclusion of Corollary \ref{corollary:iniintd-1}. The other estimate \eqref{eq:iniintd-log} also follows immediately from the exponential decay estimates of $\tilde{v}_{d,m}^\e$ and $\tilde{w}_{d,m}^\e - \tilde{w}_{d,m}^\#$, which can be deduced by $v_{d,k} - \bar{v}_{d,k} \in E(2,\bar\alpha;d,k)$ and respectively $w_{d,k} - w_{d,k}^\# \in E(2,\bar\alpha;d,k)$ for all $k\geq 0$. As noted earlier, the rest of the proof follows then the same argument of Remark \ref{remark:higher-100}, so is omitted. This finishes the proof of Theorem \ref{theorem:higher}. 
\end{proof}


\section{Further Observations}\label{section:further}

This section is devoted to making some further observations on the (higher order) convergence rates for uniformly parabolic Cauchy problems. In Section \ref{subsection:nosc}, we obtain the higher order convergence rate for \eqref{eq:ube-pde}. In Section \ref{subsection:osc}, we achieve the optimal convergence rate for \eqref{eq:ue-pde-nl} under some additional structure condition on the operator $F$ and the initial data $g$.


\subsection{Non-Oscillatory Initial Data and Higher Order Convergence Rate}\label{subsection:nosc}

Based on the construction of the higher order correctors, we are able to achieve the higher order convergence rate of the homogenization process of the problem \eqref{eq:ube-pde}. The iteration argument is basically the same with the proof of Theorem \ref{theorem:higher}. The key difference here is that we begin with the higher order error correction in the interior, not near the initial time layer. This seems to be reasonable, since the initial data of \eqref{eq:ube-pde} is not rapidly oscillatory. 

The construction of the higher order interior correctors for \eqref{eq:ube-pde} is essentially the same with Proposition \ref{proposition:int}, and has already been studied in the authors' previous work \cite{KL1} in the framework of elliptic equations. 

\begin{proposition}\label{proposition:intd-nosc} Assume that $F:\cS^n\times\R^n\times[0,T]\times\T^n\times\T\ra\R$ satisfies \eqref{eq:ellip-F} - \eqref{eq:Ck-Ca-F}. Let $\{\bar{g}_k :\R^n\ra\R\}_{k=0}^\infty$ be such that $\bar{g}_k\in S(k)$. Then there exist $\{w_k  : \R^n\times[0,T]\times\T^n\times\T\ra\R\}_{k=0}^\infty$ and $\{\bar{u}_k : \R^n\times[0,T]\ra \R\}_{k=0}^\infty$ such that $w_k \in S(2,\bar\alpha;k)$, $w_0 = w_1 = 0$, $\bar{u}_k \in S(k)$ and $\bar{u}_k(\cdot,0) = \bar{g}_k$ for any $k\geq 0$, and the following holds. Set 
\begin{equation*}
\tilde{w}_m^\e (x,t) = \sum_{k=0}^m \e^k \left( w_k\left(x,t,\frac{x}{\e},\frac{t}{\e^2}\right) + \bar{u}_k(x,t)\right).
\end{equation*}
Then one has 
\begin{equation*}
\p_t \tilde{w}_m^\e = F \left( D^2 \tilde{w}_m^\e ,x,t,\frac{x}{\e},\frac{t}{\e^2}\right) + \psi_m^\e \left(x,t,\frac{x}{\e},\frac{t}{\e^2}\right) \quad \text{in }\R^n\times(0,T),
\end{equation*}
with some $\psi_m^\e : \R^n\times[0,T]\times\T^n\times\T\ra \R$ satisfying
\begin{equation*}
\left| \psi_m^\e (x,t,y,s) \right| \leq C_m \e^{m-1},
\end{equation*}
for any $0<\e\leq\frac{1}{2}$. 
\end{proposition}

\begin{proof} The main difference of the proof here from that of Proposition \ref{proposition:int} is that the function $\bar{u}_0$ in Lemma \ref{lemma:int} is chosen by the solution to 
\begin{equation*}
\begin{cases}
\p_t \bar{u}_0 = \bar{F} (D^2 \bar{u}_0,x,t) & \text{in }\R^n\times(0,T),\\
\bar{u}_0 (x,0) = \bar{g}_0(x) & \text{on }\R^n,
\end{cases}
\end{equation*}
instead of a linear equation \eqref{eq:ubk-pde} for $k=0$. It should be stressed that the matrix corrector $X$ and the effective coefficient $\bar{A}$ are chosen to be the same as those in Section \ref{subsection:int}. We omit the rest of the proof to avoid redundant arguments. 
\end{proof} 

Equipped with Proposition \ref{proposition:intd-nosc} together with Proposition \ref{proposition:ini} and Proposition \ref{proposition:int}, we are ready to state and prove the higher order convergence rate regarding the homogenization problem of \eqref{eq:ube-pde}. 

\begin{proposition}\label{proposition:higher-nosc} Assume that $F:\cS^n\times\R^n\times[0,T]\times\T^n\times\T\ra\R$ and $g:\R^n\ra \R$ satisfy \eqref{eq:ellip-F} - \eqref{eq:Ck-C2a-g}. Under these assumptions, let $u^\e$ be the bounded viscosity solution to \eqref{eq:ube-pde} for $\e>0$. Then for each integer $d\geq 0$, there exist sequences $\{\tilde{v}_{d,k},\tilde{w}_{d,k}:\R^n\times[0,T]\times\T^n\times[0,\infty)\}_{k=0}^\infty$ and $\{\tilde{w}_{d,k}^\#:\R^n\times[0,T]\times\T^n\times\T\ra\R\}_{k=0}^\infty$ such that one has, for any $m\geq 2$, any $\e\leq\frac{1}{2}$, any $x\in\R^n$ and any $0\leq t\leq T$,
\begin{equation*}
\left| u^\e (x,t) - \sum_{d=0}^{\lfloor \frac{m}{2} \rfloor -1 } \sum_{k=0}^{m-2d} \e^{k + 2d} \left( \tilde{v}_{d,k}\left( x,t,\frac{x}{\e},\frac{t}{\e^2}\right) + \tilde{w}_{d,k}\left( x,t,\frac{x}{\e},\frac{t}{\e^2}\right) \right)  \right| \leq C_m\e^{m-1},
\end{equation*}
and in particular, for $c_m \e^2|\log \e| \leq t \leq T$,
\begin{equation*}
\left| u^\e (x,t) - \sum_{d=0}^{\lfloor \frac{m}{2} \rfloor - 1} \sum_{k=0}^{m-2d} \e^{k + 2d} \tilde{w}_{d,k}^\# \left( x,t,\frac{x}{\e},\frac{t}{\e^2}\right) \right| \leq C_m\e^{m-1},
\end{equation*}
where $c_m$ and $C_m$ depend only on $n$, $\lambda$, $\Lambda$, $\alpha$, $m$, $T$ and $K$. 
\end{proposition}

\begin{proof} Let us fix $m\geq 2$. Due to Proposition \ref{proposition:intd-nosc}, we derive that for any $0<\e\leq\frac{1}{2}$, 
\begin{equation*}
\left| u^\e (x,t) - \tilde{w}_m^\e (x,t) - \e^2 u_{0,m}^\e \right| \leq C_m\e^{m-1},
\end{equation*}
for any $x\in\R^n$ and any $0\leq t\leq T$, where $u_{0,m}^\e$ is the bounded viscosity solution to \eqref{eq:udme-pde} for $d=0$, $X_{0,k} = X_{0,k}^\# = W_k$, with $W_k$ given as in \eqref{eq:Wk}. Thus, $u_{0,m}^\e$ falls under the setting of Proposition \ref{proposition:inid}, and hence we may proceed as in the proof of Theorem \ref{theorem:higher} and achieve the desired estimates. This finishes the proof. 
\end{proof}


\subsection{General Fully Nonlinear Problem and Convergence Rate}\label{subsection:osc}

Let us begin with a short overview the homogenization process of the problem \eqref{eq:ue-pde-nl}, which can be found in \cite{AB} and \cite{M1}. First we make an additional assumption on $F$ that there is $F_* :\cS^n\times\R^n\times[0,T]\times\T^n\times\T \ra \R$ for which 
\begin{equation}\label{eq:Fe-Fast}
\e^2 F \left(\frac{1}{\e^2} P,x,t,y,s\right) \ra F_* (P,x,t,y,s)\quad\text{as }\e\ra 0, 
\end{equation}
locally uniformly for all $(P,x,t,y,s)\in (\cS^n\setminus\{0\})\times\R^n\times[0,T]\times\T^n\times\T$. Here $F_*$ is called the recession operator (corresponding to $F$). It is clear from its definition that $F_*$ also satisfies \eqref{eq:ellip-F} - \eqref{eq:concave-F}. 

Following Lemma \ref{lemma:decay} and the comments above it, we obtain a (unique) function $v:\R^n\times[0,T]\times\T^n\times[0,\infty)\ra\R$ such that $v(x,t,\cdot,\cdot)$ is the solution to 
\begin{equation}\label{eq:v-pde-nl}
\begin{cases}
v_s = F_* ( D_y^2 v, x,t, y,s ) & \text{in }\T^n\times(0,\infty),\\
v (x,t,y,0) = g(y,x) & \text{on }\T^n.
\end{cases}
\end{equation}
for each $(x,t)\in\R^n\times[0,T]$. Also it induces $\bar{v}:\R^n\times[0,T]\ra\R$ given by
\begin{equation}\label{eq:vb-nl}
\bar{v}(x,t) = \lim_{s\ra\infty} v(x,t,0,s).
\end{equation}

On the other hand, let $\bar{F}$ and $w$ be defined as in the beginning of Section \ref{subsection:cell}. Under these circumstances, the $\e$-problem \eqref{eq:ue-pde-nl} is homogenized to the following effective problem 
\begin{equation}\label{eq:ub-pde}
\begin{cases}
\bar{u}_t = \bar{F} (D^2 \bar{u},x,t) & \text{in }\R^n\times(0,T),\\
\bar{u}(x,0) = \bar{g}(x) & \text{on }\R^n,
\end{cases}
\end{equation}
according to \cite{AB} and \cite{M1}, in the sense that the viscosity solution $u^\e$ of \eqref{eq:ue-pde-nl} converges to the viscosity solution $\bar{u}$ of \eqref{eq:ub-pde} locally uniformly in $\R^n\times(0,T)$. 

The following proposition gives the optimal rate of $u^\e \ra \bar{u}$ under some additional assumptions. 

\begin{proposition}\label{proposition:opt} Assume $F : \cS^n\times\R^n\times[0,T]\times\T^n\times\T\ra\R$ and $g:\R^n\times\T^n\ra\R$ verify \eqref{eq:ellip-F} - \eqref{eq:Ck-C2a-g}. Suppose that $F_*$ satisfies, with some $0\leq \delta<1$,   
\begin{equation}\label{eq:F-Fast}
| F(P,x,t,y,s) - F_*(P,x,t,y,s) | \leq K |P|^\delta,
\end{equation}
and that $v$ and $\bar{v}$ satisfy the conclusion of Proposition \ref{proposition:decay-Ck-C2a}. Under these circumstances, let $u^\e$ and $\bar{u}$ be the viscosity solutions to \eqref{eq:ue-pde-nl} and, respectively, \eqref{eq:ub-pde}. Then there are positive constants $c$ and $C$, depending only on $n$, $\lambda$, $\Lambda$, $\alpha$, $\delta$ and $K$, such that for any $0<\e\leq\frac{1}{2}$, 
\begin{equation}\label{eq:opt}
\left| u^\e (x,t) - \bar{u} (x,t) \right| \leq C\e^{\min(1, 2-2\delta)},
\end{equation}
for all $x\in\R^n$ and $c\e^2|\log\e| \leq t\leq T$. 
\end{proposition}

\begin{remark}\label{remark:opt-rate} The inequality \eqref{eq:F-Fast} implies that 
\begin{equation}\label{eq:Fe-Fast-d}
\left| \e^2 F \left( \frac{1}{\e^2} P ,x,t,y,s \right) - F_* (P,x,t,y,s) \right| \leq K \e^{2-2\delta} |P|^\delta,
\end{equation}
for any $(P,x,t,y,s) \in\cS^n\times\R^n\times[0,T]\times\T^n\times\T$. In comparison of \eqref{eq:Fe-Fast-d} with \eqref{eq:opt}, we realize that the rate of $u^\e \ra \bar{u}$ depends sensitively on the rate of \eqref{eq:Fe-Fast}. 
\end{remark}

\begin{remark}\label{remark:opt-assump}
The second additional assumption that $v$ and $\bar{v}$ satisfy the assertion of Proposition \ref{proposition:decay-Ck-C2a} has been made because this assumption fails to hold for general $F_*$. The main reason is that nonlinear $F_*$ is Lipschitz continuous (in the matrix variable $P$) at best, which prevents us from having Proposition \ref{proposition:decay-Ck-C2a}. We shall provide some concrete example later in this regard. 
\end{remark}

\begin{proof}[Proof of Proposition \ref{proposition:opt}]
Throughout this proof, we will write by $c$ and $C$ positive constants depending at most on $n$, $\lambda$, $\Lambda$, $\alpha$, $\delta$ and $K$, and let them vary from one line to another. 

Define $\tilde{v}^\e: \R^n\times[0,T]\ra\R$ by 
\begin{equation}
\tilde{v}^\e (x,t) = v \left(x,t,\frac{x}{\e},\frac{t}{\e^2}\right) - \bar{v}(x,t). 
\end{equation} 
Since $v$ and $\bar{v}$ are assumed to satisfy \eqref{eq:decay-Ck-C2a} for all $m\geq 0$, we observe that $\tilde{v}^\e$ is a (classical solution) to 
\begin{equation*}
\begin{dcases}
\tilde{v}_t^\e = F \left( D^2 \tilde{v}^\e ,x,t,\frac{x}{\e},\frac{t}{\e^2}\right) + \psi^\e \left(x,t,\frac{x}{\e},\frac{t}{\e^2}\right) & \text{in }\R^n\times(0,T),\\
\tilde{v}^\e (x,0) = g\left(x,\frac{x}{\e} \right) - \bar{v}(x,0) & \text{on }\R^n,
\end{dcases}
\end{equation*}
with $\psi^\e: \R^n\times[0,T]\times\T^n\times[0,\infty)\ra\R$ defined by 
\begin{equation*}
\psi^\e (x,t,y,s) = F_* \left( \e^{-2} V_0 , x,t,y,s \right) - F \left( \e^{-2} V_0 + \e^{-1} V_1 + V_2  ,x,t,y,s\right),
\end{equation*}
where 
\begin{equation*}
V_0 = D_y^2 v,\quad V_1 = D_{xy} v,\quad V_2 = D_x^2 (v- \bar{v}). 
\end{equation*}

One may notice that \eqref{eq:decay-Ck-C2a} implies that all $V_0$, $V_1$ and $V_2$ satisfy the exponential decay estimate. Thus, utilizing \eqref{eq:F-Fast}, we observe that
\begin{equation}\label{eq:FastV0}
\left| F_* \left( \e^{-2} V_0 ,x,t,y,s\right) - F \left( \e^{-2} V_0 ,x,t,y,s\right) \right| \leq C\e^{-2\delta }e^{-\delta\beta s},
\end{equation}
for any $(x,y,s) \in \R^n\times\T^n\times[0,\infty)$ and any $0<\e\leq\frac{1}{2}$. On the other hand, we have from \eqref{eq:ellip-F} that
\begin{equation}\label{eq:FV0}
\left| F \left( \e^{-2} V_0 ,x,t,y,s\right)  - F \left( \e^{-2} V_0 + \e^{-1} V_1 + V_2 ,x,t,y,s \right) \right| \leq C\e^{-1} e^{-\beta s},
\end{equation}
for any $(x,t,y,s) \in \R^n\times[0,T]\times\T^n\times[0,\infty)$ and any $0<\e\leq\frac{1}{2}$. Combining \eqref{eq:FastV0} with \eqref{eq:FV0}, we arrive at 
\begin{equation*}
\left| \psi^\e (x,t,y,s) \right| \leq C\e^{-2\delta} e^{-\delta \beta s},
\end{equation*}
for any $(x,t,y,s) \in \R^n\times[0,T]\times\T^n\times[0,\infty)$ and any $0<\e\leq\frac{1}{2}$.

Thus, arguing analogously as in the proof of Lemma \ref{lemma:inid-1}, with the bounded viscosity solution $\tilde{u}^\e$ of 
\begin{equation*}
\begin{dcases}
\begin{aligned}
\tilde{u}_t^\e &=  \frac{1}{\e^2} F\left( D^2 \tilde{u}^\e + \e^{-2} V_0 x,t,\frac{x}{\e},\frac{t}{\e^2}\right)  \\
&\quad - \frac{1}{\e^2} F \left( \e^{-2} V_0,x,t,\frac{x}{\e},\frac{t}{\e^2}\right)
\end{aligned}
& \text{in }\R^n\times(0,T),\\
\tilde{u}^\e (x,0) = \bar{g}(x) & \text{on }\R^n.
\end{dcases}
\end{equation*} 
one has, for any $0<\e\leq\frac{1}{2}$, 
\begin{equation}\label{eq:opt-ini}
\left| u^\e (x,t) - \tilde{u}^\e (x,t) \right| \leq C\e^{2-2\delta},
\end{equation}
for all $x\in\R^n$ and all $c\e^2|\log\e| \leq t \leq T$.

On the other hand, we know that Proposition \ref{proposition:cell-Ck-C2a} is true under the assumptions \eqref{eq:ellip-F} - \eqref{eq:Ck-Ca-F} on $F$. Hence, it follows from the estimate \eqref{eq:cell-Ck-C2a} and the assumption \eqref{eq:decay-Ck-C2a}, which holds also for $\bar{g}$, that the solution $\bar{u}$ to \eqref{eq:ub-pde} satisfies $\bar{u} \in C^\infty(\R^n\times[0,T])$ and 
\begin{equation}\label{eq:ub-Ck}
\sum_{|\mu| + 2\nu = l}\left| D_x^\mu \p_t \bar{u} (x,t) \right| \leq C_l,
\end{equation}
for any $l  \geq 0$. Now let $w(x,t,\cdot,\cdot)$ be the unique solution of 
\begin{equation*}
\begin{dcases}
w_s = F \left( D_y^2 w + D_x^2 \bar{u}(x,t),x,t,y,s \right) - \bar{F}(D_x^2 \bar{u}(x,t),x,t)\quad\text{in }\T^n\times\T \\
w(x,t,0,0) =0,
\end{dcases}
\end{equation*}
for each $(x,t)\in\R^n\times[0,T]$. Due to Proposition \ref{proposition:cell-Ck-C2a} and \eqref{eq:ub-Ck}, we have $w\in C^\infty(\R^n\times[0,T];C^{2,\hat\alpha}(\T^n\times\T))$, for any $0<\hat\alpha<\bar\alpha$, and 
\begin{equation*}
\sum_{|\mu| + 2\nu = l}\norm{ D_x^\mu \p_t^\nu w(x,t,\cdot,\cdot) }_{C^{2,\bar\alpha}(\T^n\times\T)} \leq C_l,
\end{equation*}
for any $l  \geq 0$. 

Therefore, arguing as above, we observe that the function $\tilde{w}^\e : \R^n\times[0,T]\ra\R$, defined by
\begin{equation*}
\tilde{w}^\e (x,t) = \bar{u}(x,t) + \e^2 w \left(x,t,\frac{x}{\e},\frac{t}{\e^2}\right),
\end{equation*}
solves 
\begin{equation*}
\begin{dcases}
\tilde{w}_t^\e = F \left( D^2\tilde{w}^\e, x,t,\frac{x}{\e},\frac{t}{\e^2}\right) + \psi^\e \left(x,t,\frac{x}{\e},\frac{t}{\e^2}\right) & \text{in }\R^n\times(0,T),\\
\tilde{w}^\e (x,0) = \bar{g}(x) & \text{on }\R^n,
\end{dcases}
\end{equation*}
for some $\psi^\e : \R^n\times[0,T]\times\T^n\times\T\ra\R$ satisfying 
\begin{equation*}
\left| \psi^\e (x,t,y,s) \right| \leq C\e,
\end{equation*}
for any $0<\e\leq\frac{1}{2}$. 

Now we may proceed as in the proof of Lemma \ref{lemma:intd-1} and deduce that
\begin{equation}\label{eq:opt-int}
\left| u^\e (x,t) - \tilde{w}^\e \left(x,t,\frac{x}{\e},\frac{t}{\e^2}\right) - \e^2 u_1^\e (x,t)\right| \leq C\e,
\end{equation}
for all $x\in\R^n$ and all $0\leq t\leq T$, provided $0<\e\leq\frac{1}{2}$, where $u_1^\e$ is the bounded viscosity solution of \eqref{eq:u1me-pde}. Note that $u_1^\e$ is bounded globally, especially independent of $\e$. Finally, the error estimate \eqref{eq:opt} can be deduced by combining \eqref{eq:opt-ini} and \eqref{eq:opt-int}. 
\end{proof}

Let us finish this subsection with an example that reveals that the assumptions of Proposition \ref{proposition:opt} are satisfied for certain $F$ and $g$. 

\begin{example}\label{example:counter} Let $F_*:\cS^n\ra\R$ be such that $F_*(P) < - F_*(-P)$ for any $P\in\cS^n\setminus\{0\}$. For instance, one may take $F_*$ by Pucci's minimal operator for the lower ellipticity bound $\lambda'>\lambda$ and the upper ellipticity bound $\Lambda ' <\Lambda$. On the other hand, let $g: \R^n\times\T^n\ra\R$ be given by 
\begin{equation*}
g(x,y) = \psi(x) \phi(y),
\end{equation*} 
with $\phi:\T^n\ra\R$ and $\psi:\R^n\ra\R$ bounded and smooth.

Let us write by $F_-(P)$ and $F_+(P)$ the functionals $F_*(P)$ and, respectively, $- F_*(-P)$, and consider the spatially periodic Cauchy problem,
\begin{equation*}
\begin{cases}
\p_s v_\pm = F_\pm (D_y^2 v_\pm) & \text{in }\T^n\times(0,\infty),\\
v_\pm (y,0) = \phi(y) & \text{on }\T^n.
\end{cases}
\end{equation*} 
According to Lemma \ref{lemma:decay}, there are unique real numbers $\gamma_+$ and $\gamma_-$ such that $\gamma_\pm = \lim_{s\ra\infty} v_\pm(0,s)$. Notice that $v_\pm\in C^{2,\alpha}$ for some $0<\alpha<1$ depending only on $n$, $\lambda$ and $\Lambda$, owing to the convexity of $F_+$ and the concavity of $F_-$.

Let us observe that $\gamma_+ > \gamma_-$. First it follows from the comparison principle that $v_+ > v_-$ in $\T^n\times(0,\infty)$, which implies $\gamma_+ \geq \gamma_-$. Moreover, since $F_+ (P) > F_- (P)$ for any nonzero $P\in\cS^n$, the function $w = v_+ - v_-$ solves
\begin{equation}
\p_s (v_+ - v_-) \geq \tr (A (y,s) D_y^2 (v_+ - v_-))\quad\text{in }\T^n\times(0,\infty),
\end{equation}
where $A$ is the linearized coefficient associated with $F_+$. This implies that the function $W(s) = \min_{\T^n} (v_+(\cdot,s) - v_-(\cdot,s))$ is non-decreasing for $s>0$, whence we have $\gamma_+ > \gamma_-$.

Now let $v$ be the solution to \eqref{eq:v-pde-nl}. Then the uniqueness of $v$ implies that $v(x,y,s) = \psi(x)v_-(y,s)$ if $\psi(x)\geq 0$ and $v(x,y,s) = \psi(x)v_+(y,s)$ if $\psi(x)\leq 0$. This also implies that the function $\bar{v}$ defined by \eqref{eq:vb-nl} satisfies $\bar{v}(x) = \gamma_+\psi(x)$ if $\psi(x)\geq 0$ and $\bar{v}(x) = \gamma_-\psi(x)$ if $\psi(x)\leq 0$. 

This implies that if $\psi$ changes sign at some point, then $v$ and $\bar{v}$ are not even differentiable at that point. On the other hand, we have $v$ and $\bar{v}$ satisfying the conclusion of Proposition \ref{proposition:decay-Ck-C2a}, provided that $\psi$ is either uniformly positive or uniformly negative.  
\end{example}

\noindent
{\it Acknowledgement.}
We would like to thank the anonymous referee for the valuable comments, which significantly improved the exposition and the accuracy of the article.

\end{document}